%% file: artin_conductor.tex
\documentclass[a4paper,twoside,11pt]{amsart}

\usepackage{ae}
\usepackage{amsmath}
\usepackage{amssymb}
\usepackage{url}
\usepackage{xypic}
\usepackage{diagrams}
\usepackage{hyperref}
\usepackage{color}

\newtheorem{thm}{Theorem}[section]
\newtheorem{prop}[thm]{Proposition}
\newtheorem{lem}[thm]{Lemma}
\newtheorem{cor}[thm]{Corollary}

\theoremstyle{definition}
\newtheorem{defi}[thm]{Definition}

\theoremstyle{remark}

\newtheorem{remark}[thm]{Remark}

\newenvironment{packed_enum}{
\begin{enumerate}
  \setlength{\itemsep}{1pt}
  \setlength{\parskip}{0pt}
  \setlength{\parsep}{0pt}
}{\end{enumerate}}
\input definitions

\def\phi{\varphi}

\romanenum

\begin{document}

\title{A refinement of the Artin conductor and the base change conductor}

\author{Ching-Li Chai and Christian Kappen}

\begin{abstract}
For a local field $K$ with positive residue characteristic $p$, we introduce, in the first part of this paper, a refinement $\bAr_K$ of the classical Artin distribution $\Ar_K$. It takes values in cyclotomic extensions of $\Q$ which are unramified at $p$, and it bisects $\Ar_{K}$ in the sense that $\Ar_{K}$ is equal to the sum of $\bAr_K$ and its conjugate distribution. Compared with $\frac{1}{2}\Ar_K$, the bisection $\bAr_K$ provides a higher resolution on the level of tame ramification. In the second part of this article, we prove that the base change conductor $c(T)$ of an analytic $K$-torus $T$ is equal to the value of $\bAr_{K}$ on the $\Q_p$-rational Galois representation $X^*(T)\otimes_{\Z_p}\Q_p$ that is given by the character module $X^*(T)$ of $T$. We hereby generalize a formula for the base change conductor of an algebraic $K$-torus, and we obtain a formula for the base change conductor of a semiabelian $K$-variety with potentially ordinary reduction.
\end{abstract}

\maketitle

\pagestyle{myheadings}
\markboth{Ching-Li Chai and Christian Kappen}{A bisection of the Artin character and a formula for the base change conductor}

\tableofcontents


\section{Introduction}


The base change conductor $c(T)$ of a semiabelian variety $A$ over a local field $K$ is an interesting arithmetic invariant; it measures the growth of the volume of the Néron model of $A$ under base change, cf.\ \cite{chaiyudeshalit} \S 10. In the article \cite{chaiyudeshalit}, it is shown that if $T$ is an algebraic $K$-torus, then $c(T)$ satisfies the formula
\[
c(T)\,=\,\frac{1}{2}\Ar_K(X^*(T)_\Q)\;,\quad\quad(*)
\]
where $\Ar_K(\cdot)$ denotes the Artin distribution over $K$ and where $X^*(T)_\Q$ is the $\Q$-rational character group of $T$. In the present article, we generalize this result to a statement on analytic $K$-tori, for local fields $K$ with positive residue characteristic. As a corollary, we obtain, for such fields $K$, a formula for the base change conductor $c(A)$ of a semiabelian $K$-variety $A$ with potentially ordinary reduction; in some sense, we are using analytic $K$-tori as a link to connect $A$ with an algebraic $K$-torus. In \cite{chaisemiab} \S 6, the existence of such a formula was conjectured based on the fact that $c(A)$ is an isogeny invariant if $A$ has potentially ordinary reduction, cf.\ \cite{chaisemiab} Thm.\ 6.8. Let us note that we consider the problem of computing the base change conductor of a semiabelian $K$-variety in the special case of potentially ordinary reduction, but without making any assumptions on the dimension or the ramification behavior.

Let us assume that $K$ has positive residue characteristic $p$. To generalize formula $(*)$, we extend the classes of objects on both sides of that equation. As for the left hand side, we consider analytic $K$-tori instead of algebraic $K$-tori. These are geometric objects which can be described geometrically via rigid or uniformly rigid geometry. They also admit Galois-theoretic descriptions in terms of their character modules, which are strongly continuous $\Z_p[G_K]$-modules, i.e.\ $\Z_p[G_K]$-modules which are finite and free over $\Z_p$ and whose $G_K:=\Gal(K^\sep/K)$-actions are trivial on some subgroups of finite index. In line with this dichotomy of geometry and Galois theory, the base change conductor $c(T)$ of an analytic $K$-torus $T$ can be defined geometrically as $c_\textup{geom}(T)$, via a theory of formal Néron models for uniformly rigid spaces as it has been developed in \cite{formalnms}, or Galois-theoretically as $c_\Gal(T)$, in terms of the character module $X^*(T)$ of $T$. We will show that in fact $c_\textup{geom}(T)$ and $c_\Gal(T)$ coincide, so that both definitions of $c(T)$ are equivalent. The geometric point of view is the more natural one, and it allows further generalizations. On the other hand, the Galois-theoretic recipe is technically simpler to handle, and it is sufficient for proving the desired formula for the base change conductor of a semiabelian $K$-variety with potentially ordinary reduction.

As for the right hand side of equation $(*)$, instead of $\Q[G_K]$-modules, we consider strongly continuous $\Q_p[G_K]$-modules. However, $1/2\cdot\Ar_K$ does not compute the base change conductor for a general analytic $K$-torus $T$: to obtain a correct formula for $c(T)$, we have to replace the trivial bisection $1/2\cdot\Ar_K$ of $\Ar_K$ by a suitable nontrivial bisection, where a bisection of $\Ar_K$ is a distribution $\chi$ on the space of locally constant functions on $G_K$ such that
\[
\Ar_K\,=\,\chi\,+\,\overline{\chi}\;,
\]
with $\overline{\chi}$ denoting the conjugate of $\chi$ (in the sense explained in the first paragraph of Section \ref{notsec}). Let us observe that every such bisection $\chi$ coincides with $1/2\cdot \Ar_K$ on the characters of $\Q[G_K]$-modules. We define a specific nontrivial bisection $\bAr_K$ of $\Ar_K$, cf.\ Definition \ref{bcconddefi}, which encodes more information on tame ramification than $1/2\cdot\Ar_K$ does. It takes values in cyclotomic $p$-unramified extensions of $\Q$, and it exhibits, for strongly continuous $\Q_p[G_K]$-modules, the same behavior with regard to restriction to subgroups as $\Ar_K$, cf.\ Prop.\ \ref{barresprop}. Our definition of $\bAr_K$ relies on the nontrivial bisection
\[
\bAr_n\,=\,\frac{1}{n}\sum_{i=0}^{n-1}i\cdot(\zeta\mapsto\zeta^i)
\]
of the augmentation character of the group $\mu_n$ of $n$-th roots of unity in an algebraically closed field of characteristic zero, for varying natural numbers $n$ prime to $p$. Let us point out that our definition of $\bAr_K$ also makes sense when $K$ is of equal characteristic zero.

The main theorem \ref{mainthm} of this article is the following strengthening of formula $(*)$:

\begin{thm}\label{mainthmintro}
If $T$ is an analytic $K$-torus with character module $X^*(T)$, then
\[
c(T)\,=\,\bAr_K(X^*(T))\;.
\]
\end{thm}

Formula $(*)$ is easily seen to be equivalent to the fact that $c(\cdot)$ is an isogeny invariant for algebraic $K$-tori. While formula $(*)$ follows from Theorem \ref{mainthmintro}, we do not reprove it in this article, since the isogeny invariance of $c(\cdot)$ for algebraic $K$-tori enters into the proof of Theorem \ref{mainthmintro}.

The proof of Theorem \ref{mainthmintro} can be summarized as follows: there is a natural functor $\hat{\cdot}$ from the category $\Tori_K$ of algebraic $K$-tori to the category $\uTor_K$ of analytic $K$-tori, and if $T\in\Tori_K$, then de Shalit's recipe \cite{chaiyudeshalit} A1.9 shows that $c(T)=c_\Gal(\hat{T})$, where $c(T)$ is the classical base change conductor of $T$, which is defined in terms of algebraic Néron models, and where $c_\Gal(\hat{T})$ is the base change conductor of $\hat{T}$, defined Galois-theoretically. It is easily seen that $\hat{\cdot}$ induces an equivalence between $\uTor_K$ and the pseudo-abelian envelope of $\Tori_K$; since $c(\cdot)$ is an isogeny invariant on $\Tori_K$ by the main result of \cite{chaiyudeshalit}, it follows that $c_\Gal(\cdot)$ is an isogeny invariant on $\uTor_K$. Using Artin's theorem and the good properties of $\bAr_K$ mentioned above, the proof of Theorem \ref{mainthmintro} is now reduced to the case of a one-dimensional tame situation, which is handled explicitly. If the residue field of $K$ is algebraically closed, there is a functor $\hat{\cdot}$ from the category of semiabelian $K$-varieties with potentially ordinary reduction to $\uTor_K$, and again de Shalit's recipe \cite{chaiyudeshalit} A1.9 shows that $c(A)=c_\Gal(\hat{A})$, which yields the desired formula for the base change conductor of semiabelian varieties with potentially ordinary reduction.

This paper is structured as follows: in Section 2, we fix some notation which will be used throughout the article. In Section 3, we make no assumption on the residue characteristic of $K$; we define and study the nontrivial bisection $\bAr_{L/K}$ of $\Ar_{L/K}$, for finite Galois extensions $L/K$, see the introduction of Section \ref{bartconsec} for a more detailed discussion. We first consider the totally tamely ramified case (cf.\ \ref{tameramsec}),  then the general case (cf.\ \ref{generalramsec}). In Section \ref{uppernumbsec}, we express $\bAr_{L/K}$ in terms of the upper numbering ramification filtration of $\Gal(L/K)$; the resulting expression is then used in Section \ref{quotsec} to study the behavior of $\bAr_{L/K}$ with respect to pushforwards to quotients. Here we show that, for varying $L$, the characters $\bAr_{L/K}$ define a distribution $\bAr_K$ on the space of locally constant functions on $\Gal(K^\sep/K)$, which is then a bisection of $\Ar_K$. In Section \ref{subgrpressec}, we show that $\bAr_K$ (or rather, for $\chr\kappa_K=p>0$, its $\Gal(\Q_p^\ur/\Q_p)$-average) behaves well under restriction to subgroups, just like $\Ar_K$ does. In Section \ref{antorisec}, we discuss analytic $K$-tori, both geometrically and in terms of their character modules, assuming that $\chr\kappa_K=p>0$. In particular, we show that every analytic $K$-torus is, in the category $\uTor_K$ of analytic $K$-tori, a direct summand of an algebraic $K$-torus. In Section \ref{bccondsec}, we define the Galois-theoretic base change conductor $c_\Gal(T)$ of an analytic $K$-torus $T$ as the sum of a linear part (cf.\ \ref{linpartsec}) and a nonlinear part (cf.\ \ref{nonlinpartsec}), and we show that $c_\Gal(T)$ coincides with the geometric base change conductor $c_{\textup{geom}}(T)$ of $T$ that is defined in terms of formal Néron models. While the definition of $c_\Gal(T)$ does not require  Néron models for uniformly rigid spaces, the definition of its nonlinear part relies upon the classical Néron smoothening process, cf. Lemma \ref{welldeflem} and the subsequent Remark \ref{welldefrem}. In Corollary \ref{isoginvcor}, isogeny invariance of base change conductors is transferred from $\Tori_K$ to $\uTor_K$. In Section \ref{ratvirtsec}, we explain that for any finite Galois extension $L/K$, the isogeny invariant base change conductor on the category of $L$-split analytic $K$-tori gives rise to a rational virtual character of $\Gal(L/K)$ which is defined over $\Q_p$. The behavior of $c(\cdot)$ under Weil restriction in $\uTor_K$ or, equivalently, under induction of character modules, is discussed in Section \ref{weilressec}. In Section \ref{formulasec}, we prove our main result, Theorem \ref{mainthmintro} above: using the behavior of $c(\cdot)$ with respect to Weil restriction and the behavior of the $\Gal(\Q_p^\ur/\Q_p)$-average of $\bAr_K$ under restriction as well as Artin's theorem, we reduce, in Section \ref{redsec}, to the totally ramified cyclic case. The latter is then treated in Section \ref{cyccasesec}, by applying isogeny invariance multiple times, by using an induction argument, by invoking formula $(*)$ for algebraic $K$-tori in the totally wild case and by using the fact that the $p$-power cyclotomic polynomials are both defined over $\Q$ and irreducible over $\Q_p$. In Section \ref{abvarapplsec}, we deduce a formula for the base change conductor of a semiabelian $K$-variety $A$ with potentially ordinary reduction, and we explain how the character module $X^*(A)$ of the analytic $K$-torus attached to $A$ can be described in terms of an algebraic Hecke character in the case where $A$ is obtained from a CM abelian variety over a number field via base change to a place of potentially ordinary reduction.

\textbf{Acknowledgements.} The second author would like to express his gratitude to the University of Pennsylvania and to the NCTS of Taiwan for their hospitality and support. Moreover, he would like to thank Ulrich Görtz and Gabor Wiese for helpful discussions.

\section{Notation}\label{notsec}
Let us set up some notation which will be used throughout the paper.
If $S$ is a set, we let $|S|$ denote its cardinality. We will implicitly use basic facts on representations of finite groups in characteristic zero, for which we refer to \cite{serre_linres}. If $G$ is a finite group and if $F$ is a field of characteristic zero, then an $F[G]$-module of finite $F$-dimension will also be called an $F$-rational representation of $G$, and if $V$ is such a representation, we write $\chi_V$ to denote its character, and we write $V^G$ for the subspace of $G$-invariant vectors. We write $\textbf{\textup{r}}_G$, $\textbf{\textup{1}}_G$ and $\textbf{\textup{u}}_G=\textbf{\textup{r}}_G-\textbf{\textup{1}}_G$ for the regular character, the trivial character and the augmentation character of $G$ respectively. If $\chi$ is an $F$-valued central function on $G$, we write $\overline{\chi}$ for the $F$-valued central function on $G$ that is given by $g\mapsto\chi(g^{-1})$; if $F$ is a subfield of the field of complex numbers $\C$ and if $\chi$ is a rational virtual character of $G$, then $\overline{\chi}$ is then the complex conjugate of $\chi$. Let $C(G,F)$ denote the finite-dimensional $F$-vector space of $F$-valued central functions on $G$; we will be using the perfect pairing
\[
(\cdot|\cdot)_G\,:\,C(G,F)\times C(G,F)\rightarrow F\,;\,(\chi_1|\chi_2)\mapsto\frac{1}{|G|}\sum_{g\in G}\chi_1(g)\overline{\chi_2}(g)\;.
\]
If $\alpha:G\rightarrow G'$ is a homomorphism of finite groups and if $\chi'$ is a central function on $G'$, we write $\alpha^*\chi':=\chi'\circ\alpha$; if $\chi'$ is a character of an $F$-rational representation $\rho:G'\rightarrow\GL(V')$ of $G'$, then $\alpha^*\chi$ is the character of the $F$-rational representation of $G$ that is given by $\rho\circ\alpha$. If $\chi$ is a central function on $G$, we write $\alpha_*\chi$ for the unique central function on $G'$ with the property that the equality
\[
(\chi'|\alpha_*\chi)_{G'}\,=\,(\alpha^*\chi'|\chi)_G
\]
holds for all central functions $\chi'$ on $G'$; in other words, the measure $(\cdot|\alpha_*\chi)_{G'}$ defined by $\alpha_*\chi$ is the pushforward of the measure $(\cdot|\chi)_G$ defined by $\chi$. If $\chi$ is the character of an $F$-rational representation $V$ of $G$, then $\alpha_*\chi$ is the character of $V\otimes_{F[G]}F[G']$, where the ring homomorphism $F[G]\rightarrow F[G']$ is induced by $\alpha$. If $\alpha$ is injective, we also write $\chi'|_G$ for $\alpha^*\chi'$ and $\Ind^{G'}_G\chi$ for $\alpha_*\chi$, and we write $[G':G]$ to denote the index of $G$ in $G'$; if $\alpha$ is surjective, we also write $\Inf^G_{G'}\chi'$ for $\alpha^*\chi'$. If $V$ is a representation of $G'$ and $\alpha$ is injective, we also write $\Res_GV$ for the restriction of $V$ to $G$ via $\alpha$.

If $R$ is a ring, we write $R^\times$ for the group of units in $R$, and for $n\in\N$ we write $\mu_n(R)$ for the group of $n$-th roots of unity in $R$. If $K$ is a field, we write $K^\alg$ for an algebraic closure of $K$ and $K^\sep$ for the separable closure of $K$ in $K^\alg$. If $L/K$ is a finite field extension, we write $[L:K]$ for its degree, and if $L/K$ is Galois, we write $\Gal(L/K):=\Aut(L/K)$. If we are given a discrete valuation on $K$, we write $\nu_K:K\rightarrow\Z\cup\{\infty\}$ for the valuation map, $\O_K$ for the ring of integers of $K$, $\m_K$ for the maximal ideal in $\O_K$ and $\kappa_K$ for the residue field of $\O_K$. By a \emph{local field}, we mean a complete discretely valued field whose residue field is \emph{perfect}. If $K$ is local and if $L/K$ is a finite separable extension, we write $\mathfrak{d}_{L/K}\subseteq\O_K$ for the discriminant ideal, $f_{L/K}=[\kappa_L:\kappa_K]$ for the inertial degree and $e_{L/K}=\length_{\O_K}(\O_L/\m_K\O_L)$ for the ramification index of $L/K$. Finally, for $n\in\N$ and $r\in\Z$, we write $r\mod n$ for the natural image of $r$ in $\Z/n\Z$.

\section{A bisection of the Artin character}\label{bartconsec}

If $K$ is a local field, if $L/K$ is a finite Galois extension with Galois group $\Gamma$ and if $(\Gamma_i)_{i\geq -1}$ is the  lower numbering ramification filtration on $\Gamma$, then the Artin character $\Ar_{L/K}:\Gamma\rightarrow\Q$ is the $\Q$-valued central function
\[
\Ar_{L/K}\,=\,\sum_{i\geq 0}\frac{1}{[\Gamma_0:\Gamma_i]}\Ind^{\Gamma}_{\Gamma_i}\textbf{\textup{u}}_{\Gamma_i}\;,
\]
cf.\ \cite{serre_localfields} Chap.\ VI \S 2. It is a $\Q$-linear combination of characters of representations of $\Gamma$ that are defined over $\Q$; in other words, it is a rational virtual character of $\Gamma$ that is $\Q$-rational. If $V$ is any representation of $\Gamma$ over a field $F$ containing $\Q$ and if $\chi_V$ denotes the character of $V$, then the Artin conductor
\[
\Ar_{L/K}(V)\,:=\,(\Ar_{L/K}|\chi_V)_\Gamma\,=\,\sum_{i\geq 0}\frac{1}{[\Gamma_0:\Gamma_i]}\dim_FV/V^{\Gamma_i}
\]
of $V$ is a measure of the ramification of $V$. In the present section, we introduce a refinement of $\Ar_{L/K}$ which provides a higher resolution at the level of tame ramification. 

Let us give an overview of the constructions and results of this section: let $n$ denote the cardinality of the tame ramification group $\Gamma_0/\Gamma_1$, and let $\mu_n$ denote the group of $n$-th roots of unity in an algebraic closure $\Q^\alg$ of $\Q$. In Definition \ref{basicbadef}, we introduce a central function $\bAr_n$ on $\mu_n$ with values in $\Q(\mu_n)$ which is a nontrivial \emph{bisection} of the augmentation character, i.e.\ which satiefies
\[
\bAr_n\,+\,\overline{\bAr_n}\,=\,\textbf{\textup{u}}_{\mu_n}\;,
\]
where $\overline{\bAr_n}$ denotes the conjugate of $\bAr_n$. We write
\[
\Ar_{L/K}\,=\,\sum_{i\geq 0}\frac{1}{[\Gamma_0:\Gamma_i]}\Ind^{\Gamma}_{\Gamma_i}\textbf{\textup{u}}_{\Gamma_i}\,=\,\Ind^\Gamma_{\Gamma_0}\left(\textbf{\textup{u}}_{\Gamma_0}+\sum_{i\geq 1}\frac{1}{[\Gamma_0:\Gamma_i]}\Ind^{\Gamma_0}_{\Gamma_i}\textbf{\textup{u}}_{\Gamma_i}\right)\;,
\]
we further write
\[
\textbf{\textup{u}}_{\Gamma_0}\,=\,\Inf^{\Gamma_0}_{\Gamma_0/\Gamma_1}\textbf{\textup{u}}_{\Gamma_0/\Gamma_1}+\Ind^{\Gamma_0}_{\Gamma_1}\textbf{\textup{u}}_{\Gamma_1}\;,
\]
and we identify $\Gamma_0/\Gamma_1$ with $\mu_n$ by using the Galois action on a uniformizer of the maximally tame subextension of $L/K$, cf.\ Section \ref{generalramsec}. The nontrivial bisection $\bAr_n$ of $\textup{\textbf{u}}_{\mu_n}$ then yields a nontrivial bisection $\bAr_{L/K}$ of $\Ar_{L/K}$,
\[
\Ar_{L/K}\,=\,\bAr_{L/K}+\overline{\bAr_{L/K}}\;,
\]
which is a $\Q$-linear combination of $\Q(\mu_n)$-rational characters of $\Gamma$, cf.\ Definition \ref{bcconddefi}.

In Section \ref{subgrpressec}, we consider, for a prime number $p$ that is coprime to $n$, the $\Gal(\Q_p(\mu_n)/\Q_p)$-average
\[
\bAr_{L/K}^{(p)}\,=\,\frac{1}{[\Q_p(\mu_n):\Q_p]}\sum_{\sigma\in\Gal(\Q_p(\mu_n)/\Q_p)}\sigma\circ\bAr_{L/K}
\]
of $\bAr_{L/K}$, and we write $\bAr_{L/K}^{(0)}:=\bAr_{L/K}$. The $\bAr_{L/K}^{(p)}$, with $p$ prime as above or $p=0$, are clearly also bisections of the Artin character. We will show that these bisections enjoy the following properties, reminiscent of those of $(1/2)\Ar_{L/K}$:
\begin{enumerate}
\item If $M$ is a subextension of $L/K$ that is Galois over $K$ and if $\alpha$ is the corresponding surjection from $\Gamma$ onto the Galois group of $M$ over $K$, then
\[
\bAr_{M/K}\,=\,\alpha_*\bAr_{L/K}\;;
\]
cf.\ Prop.\ \ref{pushforwardprop}.
\item If $K$ has residue characteristic $p$, if $M$ is any subfield of $L$ containing $K$ and if $\Gamma'\subseteq\Gamma$ is the subgroup of automorphisms fixing $M$, then
\[
\bAr^{(p)}_{L/K}|_{\Gamma'}\,=\,f_{M/K}\bAr^{(p)}_{L/M}+\frac{1}{2}\nu_K(\mathfrak{d}_{M/K})\textbf{\textup{r}}_{\Gamma'}\;;
\]
cf.\ Prop.\ \ref{barresprop}.
\end{enumerate}
One may say that the bisections
\[
\frac{1}{2}\Ar_{L/K}\quad,\quad\bAr_{L/K}^{(p)}\quad\textup{and}\quad\bAr_{L/K}
\]
of $\Ar_{L/K}$ are ordered by increasing resolution, in the sense that one is obtained from the other by taking Galois averages. The bisection $\bAr_{L/K}$ has the advantage of providing high resolution of tame ramification, while it has the disadvantage that its restriction to subgroups may be difficult to describe: in general, the analog of statement ({\rm ii}) above does not hold for $\bAr_{L/K}$. Of course, the above three bisections coincide when paired with $\Q$-rational representations, and the last two still agree when paired with $\Q_p$-rational representations, for $p>0$.

\subsection{A bisection $\bAr_n$ of the cyclotomic augmentation character}\label{tameramsec}
Let us fix an algebraic closure $\Q^\alg$ of $\Q$. For each $n\in\N_{\geq 1}$, we let $\mu_n$ denote the group of $n$-th roots of unity in $\Q^\alg$, and for each $r\in\Z/n\Z$, we let 
\[
\chi_r:\mu_n\rightarrow \mu_n
\]
denote the homomorphism sending elements in $\mu_n$ to their $r^\textup{th}$ powers. Then 
\[
\{\chi_r\,;\,r\in\Z/n\Z\}\cong\Z/n\Z
\]
is the set of irreducible $\Q^\alg$-valued characters of the group $\mu_n$, and the conjugate $\overline{\chi}_r$ of a character $\chi_r$ is given by $\overline{\chi}_r=\chi_{-r}$.


\begin{defi}\label{basicbadef}
For each $n\in\N$, we let $\bA_n:\mu_n\rightarrow \Q(\mu_n)$ denote the function
\[
\bA_n\,:=\,\,\frac{1}{n}\sum_{r=0}^{n-1}r\chi_{(r\mod n)}\;.
\]
\end{defi}


For $d\in\N_{\geq 1}$, let 
\[
[d]:(\Q^\alg)^\times\rightarrow (\Q^\alg)^\times
\]
be the homomorphism rising elements to their $d^\textup{th}$ powers. For each $n\in\N_{\geq 1}$, it restricts to an epimorphism 
\[
[d]\,:\,\mu_{nd}\twoheadrightarrow \mu_n\;,
\]
and $[d]^*\chi_{(r\mod n)}=\chi_{(rd\mod nd)}$ for all $r\in\Z$.

\begin{lem}\label{bacondbasiclem}
The system $(\bA_n)_{n\in\N_{\geq 1}}$ satisfies the following relations for all $n,d\in\N_{\geq 1}$:
\begin{enumerate}
\item 
\[
(\bA_n|\textbf{\textup{1}}_{\mu_n})_{\mu_n}\,=\,(\overline{\bA_n}|\textbf{\textup{1}}_{\mu_n})_{\mu_n}=0
\]
\item
\[
\bA_n+\overline{\bA_n}=\textbf{\textup{u}}_{\mu_n}
\] 
\item
\[
\bA_n=[d]_*\bA_{nd}
\] 
\item
\[
\bA_{nd}|_{\mu_n}\,=\,\bA_n\,+\,\frac{d-1}{2}\textbf{\textup{r}}_{\mu_n}
\]
\end{enumerate}
Here $\bA_{nd}|_{\mu_n}$ denotes the restriction of $\bA_{nd}$ to the subgroup $\mu_n$ of $\mu_{nd}$.
\end{lem}
\begin{proof}
By Definition \ref{basicbadef}, we have 
\[
(\bA_n|\chi_{(r\mod n)})_{\mu_n}\,=\,r/n\quad(*)
\]
for all integers $0\leq r\leq n-1$. In particular, $(\bA_n|\textbf{1}_{\mu_n})_{\mu_n}=(\bA_n|\chi_0)_{\mu_n}=0$, and $(\overline{\bA_n}|\textbf{1}_{\mu_n})_{\mu_n}=(\bA_n|\overline{\textbf{1}_{\mu_n}})_{\mu_n}=(\bA_n|\textbf{1}_{\mu_n})_{\mu_n}=0$, which proves the first statement. To show ({\rm ii}), it is necessary and sufficient to verify that
\[
(\bA_n+\overline{\bA_n}|\chi_{r})_{\mu_n}\,=\,\begin{cases}0&\textup{if } r=0\\1&\textup{otherwise}\end{cases}
\]
for all $r\in\Z/n\Z$. For $r=0$, this equality follows from ({\rm i}).
The equality $(*)$ implies that for all integers $r$ with $1\leq r\leq n$, 
\[
(\overline{\bA_n}|\chi_{(r\mod n)})_{\mu_n}\,=\,(\bA_n|\overline{\chi_{(r\mod n)}})_{\mu_n}\,=\,(\bA_n|\chi_{(n-r\mod n)})_{\mu_n}\,=\,(n-r)/n\;;
\] 
so for $1\leq r\leq n-1$ we have
\[
(\bA_n+\overline{\bA_n}|\chi_{(r\mod n)})_{\mu_n}\,=\,\frac{1}{n}(r+(n-r))=1\;,
\]
as desired. Statement ({\rm iii}) follows from the equalities
\begin{eqnarray*}
([d]_*\bA_{nd}|\chi_{(r\mod n)})_{\mu_n}&=&(\bA_{nd}|[d]^*\chi_{(r\mod n)})_{\mu_{nd}}\\
&=&(\bA_{nd}|\chi_{(rd\mod nd)})_{\mu_{nd}}\\
&=&(rd)/(nd)\,=\,r/n
\end{eqnarray*}
for $0\leq r\leq n-1$. Finally, statement ({\rm iv}) holds because
\[
\bA_{nd}\,=\,\frac{1}{nd}\sum_{r=0}^{d-1}\sum_{s=0}^{n-1}(rn+s)\chi_{(rn+s\mod nd)}
\]
and because $\chi_{(rn+s\mod nd)}|_{\mu_n}=\chi_{(s\mod n)}$. Indeed, it follows that
\[
\bA_{nd}|_{\mu_n}\,=\,\left(\frac{1}{nd}\sum_{s=0}^{n-1}ds\chi_{(s\mod n)}\right)+\left(\frac{1}{nd}\sum_{r=0}^{d-1}rn\sum_{s=0}^{n-1}\chi_{(s\mod n)}\right)\;,
\]
where the first summand is $\bA_n$ and where the second summand is
\[
\frac{d-1}{2}\textbf{\textup{r}}_{\mu_n}\;,
\]
as desired.
\end{proof}

\begin{remark}
Lemma \ref{bacondbasiclem} ({\rm iii}) shows that the collection $(\bA_n)_{n\in\N_{\geq 1}}$ defines a distribution on the set of locally constant $\Q^\alg$-valued functions on 
\[
\hat{\Z}(1)=\plim_n\mu_n\;,
\]
where the above projective limit is taken over the surjective transition homomorphisms $[d]$. 

\end{remark}

\subsection{A bisection $\bAr_{L/K}$ of the Artin character}\label{generalramsec}
Let $K$ be a local field, let $L/K$ be a finite Galois extension, and let $(\Gamma_i)_{i\geq -1}$ be the lower numbering ramification filtration on $\Gamma=\Gal(L/K)$, cf.\ \cite{serre_localfields} Chap.\ IV. If $\pi_L$ is any uniformizer of $L$, then the wild ramification subgroup $\Gamma_1$ is, by definition, the kernel of the homomorphism
\[
\Gamma_0\rightarrow k_L^\times\quad,\quad\sigma\mapsto\frac{\sigma(\pi_L)}{\pi_L}\;\mod \m_L
\]
which is independent of the choice of $\pi_L$ (cf.\ \cite{serre_localfields} Chap.\ IV \S 2 Prop.\ 7) and which induces an isomorphism from $\Gamma_0/\Gamma_1$ onto a finite subgroup of $k_L^\times$. Let $n$ denote the cardinality of $\Gamma_0/\Gamma_1$; then $n$ is prime to the characteristic of $\kappa_L$, and the image of the above homomorphism is given by $\mu_n(\kappa_L)$. 

We will consider an isomorphism from $\Gamma_0/\Gamma_1$ onto $\mu_n(\kappa_L)$ that is, in general, different from the one above. To this end, let us consider the extensions
\[
K\subseteq K^\ur\subseteq K^t\subseteq L\;,
\]
where $K^\ur/K$ is the biggest unramified subextension of $L/K$ and where $K^t/K$ is the biggest tamely ramified subextension of $L/K$. Then $K^\ur=L^{\Gamma_0}$, $K^t=L^{\Gamma_1}$, and $\Gamma_0/\Gamma_1$ is naturally identified with $\Gal(K^t/K^\ur)$. Let $\pi_t$ be any uniformizer of $K^t$; then the homomorphism
\[
\Gamma_0\rightarrow k_L^\times\quad,\quad\sigma\mapsto\frac{\sigma(\pi_t)}{\pi_t}\;\mod \m_{K^t}
\]
is again independent of the choice of $\pi_t$ (cf.\ [ibid.]), and it defines an isomorphism $\Psi'_{L/K}$ of $\Gamma_0/\Gamma_1$ onto $\mu_n(\kappa_{K^t})=\mu_n(\kappa_L)$.

We fix, once and for all, an algebraic closure $K^\alg$ of $K$; by algebraic extensions of $K$ we will implicitly mean subextensions of $K^\alg/K$. If $K$ has characteristic zero, we fix an embedding of $\Q^\alg$ into $K^\alg$; otherwise, we choose an embedding of $\Q^\alg$ into an algebraic closure of the fraction field of the ring of Witt vectors of $K$. The Teichmüller homomorphism then yields an identification of $\mu_n(\kappa_L)$ with $\mu_n=\mu_n(\Q^\alg)$. Let $\Psi_{L/K}$ denote the resulting isomorphism 
\[
\Psi_{L/K}\,:\,\Gamma_0/\Gamma_1\overset{\sim}{\longrightarrow}\mu_n
\]
that is obtained from $\Psi'_{L/K}$. Clearly $\Psi_{L/K}=\Psi_{L/K^\ur}=\Psi_{K^t/K}=\Psi_{K^t/K^\ur}$.

\begin{defi}\label{bcconddefi}
In the above situation, we define
\[
\bAr_{L/K}^t\,:=\,\Psi_{L/K}^*\bA_n\;,
\]
and we set
\[
\bAr_{L/K}\,:=\,\Ind^{\Gamma}_{\Gamma_0}\left(\Inf^{\Gamma_0}_{\Gamma_0/\Gamma_1}\bA_{L/K}^t\,+\,\frac{1}{2}\Ind^{\Gamma_0}_{\Gamma_1}\textbf{\textup{u}}_{\Gamma_1}+\frac{1}{2}\sum_{i\geq 1}\frac{1}{[\Gamma_0:\Gamma_i]}\Ind^{\Gamma_0}_{\Gamma_i}\textbf{\textup{u}}_{\Gamma_i}\right)\;;
\]
this is a rational virtual character $\bAr_{L/K}$ of $\Gamma$ that is defined over $\Q(\mu_n)$, and it is called the refined Artin character of $L/K$.
\end{defi}

\begin{remark}Let us note the following basic properties of $\bAr_{L/K}$:
\begin{enumerate}
\item By Lemma \ref{bacondbasiclem} ({\rm ii}), we have
\[
\bAr_{L/K}+\overline{\bAr_{L/K}}=\Ar_{L/K}\;,
\]
so $\bAr_{L/K}$ is a bisection of $\Ar_{L/K}$. 
\item By definition, we have canonical identifications
\[
\bAr^t_{L/K}\,=\,\bAr^t_{L/K^\ur}\,=\,\bAr^t_{K^t/K}\,=\,\bAr^t_{K^t/K^\ur}\;.
\]
\item Moreover, it is clear from the definitions that if $\iota$ is the inclusion of $\Gamma_0$ into $\Gamma$, then
\[
\bAr_{L/K}\,=\,\iota_*\bAr_{L/K^\ur}\;.
\]
\end{enumerate}
\end{remark}
\subsection{Expression of $\bAr_{L/K}$ in terms of the upper numbering ramification filtration}\label{uppernumbsec}

In order to study the behavior of $\bAr_{L/K}$ with respect to pushforward to quotients of $\Gamma$, it is convenient to express $\bAr_{L/K}$ in terms of the upper numbering ramification filtration 
\[
(\Gamma^i)_{i\geq -1}
\]
on $\Gamma$. Let $g_0$ denote the cardinality of $\Gamma_0$, that is the ramification index of $L/K$.

\begin{lem}\label{uppernumberinglem}
We have
\[
\bA_{L/K}\,=\,\Ind^\Gamma_{\Gamma^0}\left(\Inf^{\Gamma^0}_{\Gamma^0/\Gamma^{1/g_0}}(\bA_{L/K}^t)\,+\,\frac{1}{2}\Ind^{\Gamma^0}_{\Gamma^{1/g_0}}\textup{\textbf{u}}_{\Gamma^{1/g_0}}\,+\,\frac{1}{2g_0}\sum_{i\geq 1}\Ind^{\Gamma^0}_{\Gamma^{i/g_0}}\textbf{\textup{u}}_{\Gamma^{i/g_0}}\right)\;.
\]
\end{lem}
\begin{proof}
Let us compute the right hand side of the desired equality. For any real number $r\geq0$, let us set
\[
\textup{\textbf{v}}_r\;:=\;\Ind^{\Gamma^0}_{\Gamma^r}\textup{\textbf{u}}_{\Gamma^r}\quad\textup{and}\quad\textup{\textbf{u}}_r\;:=\;\Ind^{\Gamma_0}_{\Gamma_r}\textup{\textbf{u}}_{\Gamma_r}\;.
\]
 The upper numbering filtration does not jump in open intervals of the form $(i/g_0,(i+1)/g_0)$ where $i\geq 0$ is an integer; hence, for any $\gamma\in\Gamma$, the function $\textbf{\textup{v}}_{r/g_0}(\gamma)$ of $r$ is constant on the open intervals of the form $(i,i+1)$, with $i\in\Z_{\geq 1}$. Let $\psi$ denote the Herbrand function of \cite{serre_localfields} Chap.\ IV \S 3; then $\textbf{\textup{v}}_{r}=\textbf{\textup{u}}_{\psi(r)}$ for any $r\in\R_{\geq -1}$. By what we just said, we can rewrite the sum 
\[
\frac{1}{g_0}\sum_{i\geq 1}\textbf{\textup{v}}_{i/g_0}(\gamma)
\]
as an integral
\[
\frac{1}{g_0}\int_{r\geq 0}\textbf{\textup{v}}_{r/g_0}(\gamma)\,dr\,=\,\frac{1}{g_0}\int_{r\geq 0}\textbf{\textup{u}}_{\psi(r/g_0)}(\gamma)\,dr\;,
\]
and by the chain rule this integral equals
\[
\frac{1}{g_0}\int_{s\geq 0} \textbf{\textup{u}}_s(\gamma)\frac{g_0}{\psi'(\phi(s))}\,ds\,=\,\int_{s\geq 0} \frac{\textbf{\textup{u}}_s(\gamma)}{\psi'(\phi(s))}\;,
\]
where $\psi'$ denotes the derivative of $\psi$. The lower numbering ramification filtration does not jump on open intervals $(i,i+1)$ with $i\in\Z_{\geq 0}$, and for $i\geq 1$ and $s\in (i-1,i)$, we have
\[
\psi'(\phi(s))=[\Gamma_0:\Gamma_i]\;;
\]
it follows that
\[
\int_{s\geq 0} \frac{\textbf{\textup{u}}_s(\gamma)}{\psi'(\phi(s))}\,=\,\sum_{i\geq 1}\frac{1}{[\Gamma_0:\Gamma_i]}\textup{\textbf{u}}_i(\gamma)\;.
\]
The desired equality now follows from the fact that 
\[
\Gamma^0=\Gamma_0\quad\textup{and}\quad\Gamma^{1/g_0}\,=\,\Gamma_1\;.
\]
\end{proof}

\subsection{Behavior of $\bAr_{L/K}$ with respect to pushforward to quotients}\label{quotsec}

We now show that the refined Artin characters of finite Galois extensions of $K$ define a distribution on the space of locally constant central functions on the absolute Galois group of $K$. Let us first note the following elementary lemma:

\begin{lem}\label{pushforwardelementarylem}
Let $G$ be a finite group, let $M,N\subseteq G$ be normal subgroups, let
\[
\xymatrix{
G\ar[r]^\alpha\ar[d]^\gamma&G/M\ar[d]^\beta\\
G/N\ar[r]^\delta&G/(MN)
}
\]
be the resulting canonical commutative diagrams of surjective homomorphisms, and let $\chi$ be a central function on $G/M$ with values in some field of characteristic zero; then
\[
\gamma_*\alpha^*\chi\,=\,\delta^*\beta_*\chi\;.
\]
\end{lem}
\begin{proof}
Let $c$ be an element of $G/N$; then
\[
(\delta^*\beta_*\chi)(c)\,=\,\frac{1}{|\ker\beta|}\sum_{\beta(b)=\delta(c)}\chi(b)\;,
\]
while
\[
(\gamma_*\alpha^*\chi)(c)\,=\,\frac{1}{|\ker\gamma|}\sum_{\gamma(a)=c}\chi(\alpha(a))\;.
\]
In other words, let $g$ be an element of $G$; then the first line yields
\[
(\delta^*\beta_*\chi)(g\mod N)\,=\,\frac{1}{[MN:M]}\sum_{h\in MN/M}\chi((g+h)\mod M)\;,
\]
while the second line yields
\[
(\gamma_*\alpha^*\chi)(g\mod N)\,=\,\frac{1}{|N|}\sum_{h\in N}\chi((g+h)\mod M)\;.
\]
We now see that the two expressions are equal: Indeed, we have a natural homomorphism
\[
N\hookrightarrow MN\rightarrow MN/M
\]
which is surjective and whose kernel is $M\cap N$. Clearly the value of $\chi((g+h)\mod M)$, for $h\in N$, only depends on the image of $h$ in $MN/M$, and $|M\cap N|\cdot|MN/M|=|N|$, which shows the desired statement. 
\end{proof}

\begin{prop}\label{pushforwardprop}
Let $M$ be a subfield of $L$ containing $K$ such that $M/K$ is Galois, and let $\alpha$ denote the natural surjection from $\Gamma$ onto the Galois group of $M$ over $K$; then
\[
\alpha_*\bA_{L/K}\,=\,\bA_{M/K}\;.
\]
\end{prop}
\begin{proof}
%
Let $\Gamma'$ denote the Galois group of $M/K$, and let $\iota:\Gamma_0\hookrightarrow\Gamma$ and $\iota':\Gamma_0'\hookrightarrow\Gamma'$ denote the inclusions of the inertia subgroups. We have a natural commutative diagram
\[
\xymatrix{
\Gamma_0\ar@{^{(}->}[r]^\iota\ar@{->>}[d]_{\alpha_0}&\Gamma\ar@{->>}[d]^\alpha\\
\Gamma_0'\ar@{^{(}->}[r]^{\iota'}&\Gamma'\;,
}
\]
and $\bAr_{L/K}=\iota_*\bAr_{L/K^t}$, $\bAr_{M/K}=\iota'_*\bAr_{M/M\cap K^t}$. In order to show the statement of the proposition, we may and do thus assume that $L/K$, and, hence, $M/K$, are both totally ramified. Let us now use the expression of $\bAr_{L/K}$ and $\bAr_{M/K}$ in terms of the upper numbering ramification filtrations given by Lemma \ref{uppernumberinglem}. Let us fix an integer $i\geq 1$, and let $\beta$ denote the inclusion $\Gamma^{i/g}\hookrightarrow\Gamma$. Then
\[
\Ind^\Gamma_{\Gamma^{i/g}}\textup{\textbf{u}}_{\Gamma^{i/g}}\,=\,\beta_*\textup{\textbf{u}}_{\Gamma^{i/g}}\;,
\]
and we are interested in computing $\alpha_*\beta_*\textup{\textbf{u}}_{\Gamma^{i/g}}=(\alpha\beta)_*\textup{\textbf{u}}_{\Gamma^{i/g}}$. Let $\beta'$ denote the inclusion of $(\Gamma')^{i/g}$ into $\Gamma'$; we claim that
\[
\alpha_*\beta_*\textup{\textbf{u}}_{\Gamma^{i/g}}\,=\,\beta'_*\textbf{\textup{u}}_{(\Gamma')^{i/g}}\;.
\]
And indeed, if $\chi$ is any character of $\Gamma'$ with underlying representation space $V$, then
\[
(\textup{\textbf{u}}_{\Gamma^{i/g}},(\alpha\beta)^*\chi)\,=\,\dim V/V^{\Gamma^{i/g}}\,=\,\dim V/V^{(\Gamma')^{i/g}}\,=\,(\textbf{\textup{u}}_{(\Gamma')^{i/g}},(\beta')^*\chi)
\]
as desired, where $\Gamma$ acts on $V$ via the surjection $\alpha:\Gamma\rightarrow\Gamma'$. Let us write $g=[L:K]$ and $g':=[M:K]$; it now follows that
\begin{eqnarray*}
\alpha_*\frac{1}{g}\sum_{i\geq 1}\Ind^\Gamma_{\Gamma^{i/g}}\textbf{\textup{u}}_{\Gamma^{i/g}}&=&\frac{1}{g}\sum_{i\geq 1}\Ind^{\Gamma'}_{(\Gamma')^{i/g}}\textbf{\textup{u}}_{(\Gamma')^{i/g}}\\
&=&\frac{1}{g/g'}\cdot\frac{1}{g'}\sum_{i\geq 1}\Ind^{\Gamma'}_{(\Gamma')^{i/(g'(g/g'))}}\textbf{\textup{u}}_{(\Gamma')^{i/(g'(g/g'))}}\\
&=&\frac{1}{g'}\sum_{i\geq 1}\Ind^{\Gamma'}_{(\Gamma')^{i/g'}}\textbf{\textup{u}}_{(\Gamma')^{i/g'}}\;,
\end{eqnarray*}
where for the last step we have used that the upper numbering filtration on $\Gamma'$ does not jump on open intervals of the form $(i/g',(i+1)/g')$.

By Lemma \ref{uppernumberinglem}, it now suffices to show that
\[
\alpha_*(\Inf^\Gamma_{\Gamma/\Gamma^{1/g}}(\bAr^t_{L/K})+\frac{1}{2}\Ind^\Gamma_{\Gamma^{1/g}}\textup{\textbf{u}}_{\Gamma^{1/g}})\,=\,\Inf^{\Gamma'}_{\Gamma'/(\Gamma')^{1/g'}}(\bAr^t_{M/K})+\frac{1}{2}\Ind^{\Gamma'}_{(\Gamma')^{1/g'}}\textup{\textbf{u}}_{(\Gamma')^{1/g'}}\;.
\]
By what we have seen so far, 
\[
\alpha_*\Ind^\Gamma_{\Gamma^{1/g}}\textbf{\textup{u}}_{\Gamma^{1/g}}=\Ind^{\Gamma'}_{(\Gamma')^{1/g}}\textbf{\textup{u}}_{(\Gamma')^{1/g}}\;, 
\]
and $(\Gamma')^{1/g}=(\Gamma')^{1/g'}$, so it suffices to show that
\[
\alpha_*(\Inf^\Gamma_{\Gamma/\Gamma_1}(\bAr^t_{L/K}))\,=\,\Inf^{\Gamma'}_{\Gamma'/\Gamma'_1}(\bAr^t_{M/K})\;.
\]

To this end, let $\gamma:\Gamma\rightarrow\Gamma/\Gamma_1$ and $\gamma':\Gamma'\rightarrow\Gamma'/\Gamma'_1$ be the natural projections, and let us write $n=|\Gamma/\Gamma_1|$, $n'=|\Gamma'/\Gamma'_1|$. We have to show that
\[
\alpha_*\gamma^*\Psi^*_{L/K}\bAr_n\,=\,(\gamma')^*\Psi^*_{M/K}\bAr_{n'}\;.\quad\quad(*)
\]
The right square in the diagram
\[
\xymatrix{
\Gamma\ar[d]_\alpha\ar[r]^\gamma&\Gamma_0/\Gamma_1\ar[r]^{\Psi_{L/K}}\ar[d]_{\alpha\mod\Gamma_1}&\mu_n\ar[d]^{[d]}\\
\Gamma'\ar[r]^{\gamma'}&\Gamma_0'/\Gamma_1'\ar[r]^{\Psi_{M/K}}&\mu_{n'}\;,
}\]
commutes: indeed, if $K^t$ denotes the maximal tamely ramified subextension of $L/K$, then $K^t\cap M$ is the maximal tamely ramified subextension of $M/K$, and the extension $K^t/K^t\cap M$ is both totally and tamely ramfied, of degree $d=n/n'$; by \cite{lang_ant} Chap.\ II \S 5 Prop.\ 12, there exists a uniformizer $\pi$ of $K^t$ such that $\pi^d$ is a uniformizer of $K^t\cap M$, which shows the desired commutativity statement. Equality ($*$) now follows from Lemma \ref{bacondbasiclem} ({\rm iii}) together with Lemma \ref{pushforwardelementarylem}, applied to the central square in the above diagram.
\end{proof}

\begin{cor}
Let $K^\alg$ be an algebraic closure of $K$, and let $\mathcal{C}$ denote the set of finite Galois extensions of $K$ within $\mathcal{C}$; then the system $(\bAr_{L/K}|\cdot)_{L/K\in\mathcal{C}}$  defines a distribution $\bAr_K(\cdot)$ on the space of locally constant central functions on the absolute Galois group $G_K=\Gal(K^\alg/K)$ of $K$.
\end{cor}

\subsection{Behavior of $\bAr_{L/K}^{(p)}$ with respect to restriction to subgroups}\label{subgrpressec}
If $\chi$ is a rational virtual character of $\Gal(L/K)$ and if $p$ is a prime number that is coprime to $n$, we let
\[
\chi^{(p)}\,:=\,\frac{1}{[\Q_p(\mu_n):\Q_p]}\sum_{\sigma\in\Gal(\Q_p(\mu_n)/\Q_p)}\sigma\circ\chi
\]
denote the $\Gal(\Q_p(\mu_n)/\Q_p)$-average of $\chi$; moreover, we set
\[
\chi^{(0)}\,:=\,\chi\;.
\]
We now study the behavior of $\bAr_{L/K}^{(p)}$ under restriction to a subgroup $\Gamma'$ of $\Gamma$, for $p=\chr\kappa_K$. 

\begin{prop}\label{barresprop}
If $M\subseteq L$ is the field whose elements are fixed by $\Gamma'$, then 
\[
\bAr^{(p)}_{L/K}|_{\Gamma'}\,=\,f_{M/K}\bAr^{(p)}_{L/M}+\frac{1}{2}\nu_K(\mathfrak{d}_{M/K})\textbf{\textup{r}}_{\Gamma'}\;.
\]
\end{prop}
\begin{proof}
Let us first reduce to the case where $L/K$ is totally ramified. We have a commutative diagram of inclusions
\[
\xymatrix{
\Gamma_0\ar[r]^\beta&\Gamma&\\
\Gamma_0'\ar[r]^{\beta'}\ar[u]^{\alpha_0}&\Gamma'\ar[u]^\alpha&\;,
}
\]
where $\Gamma_0'=\Gamma_0\cap\Gamma'$. By \cite{serre_linres} Prop.\ 22, if $\chi$ is a central function on $\Gamma_0$ with values in a field of characteristic zero, then
\[
\alpha^*\beta_*\chi\,=\,[\Gamma:\Gamma_0\Gamma']\beta'_*\alpha_0^*\chi\;;
\]
moreover, $[\Gamma:\Gamma_0\Gamma']=f_{M/K}$. Let us assume that the statement of the proposition holds in the totally ramified case; then
\begin{eqnarray*}
\bAr_{L/K}^{(p)}|_{\Gamma'}&=&\alpha^*\bAr_{L/K}^{(p)}\\
&=&\alpha^*\beta_*\bAr_{L/K^\ur}^{(p)}\\
&=&f_{M/K}\beta'_*\alpha_0^*\bAr_{L/K^\ur}^{(p)}\\
&=&f_{M/K}\beta'_*\left(\bAr^{(p)}_{L/MK^\ur}+\frac{1}{2}\nu_{K^\ur}(\mathfrak{d}_{MK^\ur/K^\ur})\textup{\textbf{r}}_{\Gamma_0'}\right)\\
&=&f_{M/K}\bAr^{(p)}_{L/M}+f_{M/K}\frac{1}{2}\nu_K(\mathfrak{d}_{MK^\ur/K^\ur})\textbf{\textup{r}}_{\Gamma'}\;.\quad(*)
\end{eqnarray*}
By \cite{serre_localfields} Chap.\ III \S 4 Prop. 8, we have
\[
\mathfrak{d}_{MK^\ur/K}=N_{K^\ur/K}\mathfrak{d}_{MK^\ur/K^\ur}\;,
\]
which shows that 
\[
\nu_K(\mathfrak{d}_{MK^\ur/K})\,=\,f_{L/K}\nu_K(\mathfrak{d}_{MK^\ur/K^\ur})\;.
\]
On the other hand, [ibid.] shows that
\[
\mathfrak{d}_{MK^\ur/K}\,=\,[MK^\ur:M]\mathfrak{d}_{M/K}\,=\,f_{MK^\ur/M}\mathfrak{d}_{M/K}\;,
\]
which implies that
\[
\nu_K(\mathfrak{d}_{MK^\ur/K})\,=\,f_{MK^\ur/M}\nu_K(\mathfrak{d}_{M/K})\;.
\]
Combining these two results, we obtain $f_{L/K}\nu_K(\mathfrak{d}_{MK^\ur/K^\ur})=f_{MK^\ur/M}\nu_K(\mathfrak{d}_{M/K})$. Since $f_{MK^\ur/M}=f_{L/M}$ and since $f_{L/K}=f_{L/M}f_{M/K}$, we get
\[
f_{M/K}\nu_K(\delta_{MK^\ur/K^\ur})=\nu_K(\delta_{M/K})\;.
\]
Plugging this into the expression ($*$), we obtain the statement of the proposition. We have thus reduced to the case where $L/K$ is totally ramified.

Let us now write $\bAr_{L/K}$ in terms of $\Ar_{L/K}$ as
\begin{eqnarray*}
\bAr_{L/K}&=&\frac{1}{2}\Ar_{L/K}+\Inf^{\Gamma}_{\Gamma/\Gamma_1}\Psi_{L/K}^*(\bAr_n-\frac{1}{2}\textbf{\textup{u}}_{\mu_n})\\
&=&\frac{1}{2}\Ar_{L/K}+\frac{1}{2}\Inf^{\Gamma}_{\Gamma/\Gamma_1}\Psi_{L/K}^*(\bAr_n-\overline{\bAr}_n)\;.
\end{eqnarray*}
By \cite{serre_localfields} Chap.\ VI \S 2 Prop.\ 4, we know that 
\[
\Ar_{L/K}|_{\Gamma'}\,=\,f_{M/K}\Ar_{L/M}+\nu_K(\mathfrak{d}_{M/K})\textbf{\textup{r}}_{\Gamma'}\;.
\]
Let us now consider the diagram
\[
\xymatrix{
\Gamma\ar@{^{(}->}[r]^\gamma&\Gamma/\Gamma_1\ar@{^{(}->}[r]^{\Psi_{L/K}}&\mu_n\\
\Gamma'\ar@{^{(}->}[r]^{\gamma'}\ar@{^{(}->}[u]_{\alpha}&\Gamma'/\Gamma_1'\ar@{^{(}->}[r]^{\Psi_{L/M}}\ar@{^{(}->}[u]_{\alpha_t}&\mu_{n'}\ar@{^{(}->}[u]_{\alpha_t'}\;,
}
\]
where $n'$ denotes the cardinality of $\Gamma'/\Gamma'_1$ and where $\alpha_t'$ is the unique injective homomorphism making the diagram commute. By the above, we must show that
\[
\alpha^*\gamma^*\Psi_{L/K}^*(\bAr_n^{(p)}-\overline{\bAr}_n^{(p)})\,=\,(\gamma')^*\Psi_{L/M}^*(\bAr_{n'}^{(p)}-\overline{\bAr}_{n'}^{(p)})\;.
\]
For this, it suffices to show that
\[
\bAr_{n'}^{(p)}-\overline{\bAr}_{n'}^{(p)}\,=\,(\alpha_t')^*(\bAr_n^{(p)}-\overline{\bAr}_n^{(p)})\;.
\]
The image of $\alpha_t'$ must coincide with the subgroup $\mu_{n'}$ of $\mu_n$; hence we can write
\[
\alpha'_t=\iota\circ\alpha''_t\;,
\]
where $\iota:\zeta\mapsto\zeta$ is the canonical inclusion of $\mu_{n'}$ into $\mu_n$ and where $\alpha''_t$ is an automorphism of $\mu_{n'}$. By Lemma \ref{bacondbasiclem} ({\rm iv}) and since $\textbf{\textup{r}}_{\mu_{n'}}=\overline{\textbf{\textup{r}}}_{\mu_{n'}}$, we know that
\[
\iota^*(\bAr_n-\overline{\bAr}_n)\,=\,\bAr_{n'}-\overline{\bAr}_{n'}\;.
\]
Moreover, if $\chi$ is any rational virtual character on $\mu_n$, then
\[
\iota^*(\chi^{(p)})\,=\,(\iota^*\chi)^{(p)}\;.
\]
Indeed, this is trivial for $p=0$; for $p>0$ let us write $r=[\Q_p(\mu_n):\Q_p]$, $r'=[\Q_p(\mu_{n'}):\Q_p]$, $r=dr'$, $G=\Gal(\Q_p(\mu_n):\Q_p)$ and $H=\Gal(\Q_p(\mu_n)/\Q_p(\mu_{n'}))$, and let $S\subseteq G$ be a system of representatives for $G/H$; then $|S|=d$, and for $\zeta\in\mu_{n'}$, we have
\[
\frac{1}{r}\sum_{\sigma\in G}\sigma(\zeta)\,=\,\frac{1}{r}\sum_{\tau\in H}\sum_{\sigma\in S}\sigma\tau(\zeta)\,=\,\frac{1}{r'}\sum_{\sigma\in G/H}\sigma(\zeta)\;.
\]
It thus suffices to show that
\[
\bAr_{n'}^{(p)}-\overline{\bAr}_{n'}^{(p)}\,=\,(\alpha_t'')^*(\bAr_{n'}^{(p)}-\overline{\bAr}_{n'}^{(p)})\;.
\]
Let us first consider the case $p>0$; then $n'$ is coprime to $p$ since $n'|n$. By \cite{serre_localfields} Chap.\ IV \S 4 Prop.\ 16, $\Gal(\Q_p(\mu_{n'})/\Q_p)$ is cyclic, generated by an element $\sigma$ whose action on $\mu_{n'}$ is given by $\zeta\mapsto \zeta^p$. It thus suffices to show that $\alpha_t''$ is given by $\zeta\mapsto\zeta^{p^s}$ for some natural number $s$. And indeed, since $L/K^t$ is totally wildly ramified, so is $K^tM/K^t$, where the composite field $K^tM$ is the maximal tamely ramified subfield of the extension $L/M$. That is, $e:=[K^tM:K^t]=p^s$ for some integer $s$. Let $\pi'_t$ be a uniformizer for $K^tM$; then by \cite{serre_localfields} Chap. I \S 6 Prop.\ 18, $\pi'_t$ satisfies an equation
\[
X^e+a_1X^{e-1}+\ldots+a_n=0
\]
with $a_i\in\O_{K^t}$, $\pi_t|a_i$, $\pi_t^2\nmid a_n$. Now $\pi_t:=a_n$ is a uniformizer for $K^t$, and in $(MK^t)^\times$ we obtain a congruence
\[
(\pi'_t)^e\equiv\pi_t\mod 1+\m_{K^tM}\;;
\]
in other words,
\[
\frac{(\pi'_t)^e}{\pi_t}\,\in\,1+\m_{K^tM}\;.
\]
Now for $\sigma\in\Gamma'/\Gamma_1'$, we have
\begin{eqnarray*}
\Psi_{L/K}(\sigma)&=&\frac{\sigma(\pi_t)}{\pi_t}\mod 1+\m_{K^t}\\
&=&\frac{\sigma(\pi_t)}{\pi_t}\mod 1+\m_{K^tM}\\
&=&\frac{\sigma((\pi_t')^e)}{(\pi_t')^e}\mod 1+\m_{K^t}\\
&=&\Psi_{L/M}(\sigma)^e\;,
\end{eqnarray*}
which shows that $\alpha''_t$ maps $\zeta$ to $\zeta^e=\zeta^{p^s}$, as desired. If $p=0$, then $L/K$ is totally tamely ramified, and hence $K^tM=K^t=L$; the map $\alpha_t''$ is thus the identity, and there is nothing to show.
\end{proof}

\section{Analytic tori}\label{antorisec}

For any field $K$ and any prime number $p$, we introduce a category $\Mod(G_K,\Z_p)$ of certain integral, $\Z_p$-rational $G_K$-representations. They arise naturally from algebraic $K$-tori. In fact, the category $\Tori_K$ of algebraic $K$-tori is naturally equivalent to the category of character groups $\Mod(G_K,\Z)$, and $\Mod(G,\Z_p)$ is naturally equivalent to the pseudo-abelian envelope of $\Mod(G_K,\Z)\otimes_\Z\Z_p$. If $K$ is local with $\chr\kappa_K=p$, then $\Mod(G_K,\Z_p)$ is naturally equivalent to a category of semiaffinoid $K$-groups, the category $\aTor_K$ of analytic $K$-tori, cf.\ Remark \ref{anintrem}. If in addition $\kappa_K$ is algebraically closed, then objects in this category are naturally obtained from semiabelian $K$-varieties with potentially ordinary reduction, cf.\ the first paragraph of Section \ref{abvarapplsec}. 

\begin{defi} \label{charmoddefi}
If $R$ is a ring and if $G$ is a group, we let $\Mod(G,R)$ denote the category of finite free $R$-modules $M$ equipped with an $R$-linear $G$-action that factors through a finite quotient of $G$; such a quotient is called a \emph{splitting quotient} for $M$. A morphism $\phi$ in $\Mod(G,R)$ is called an \emph{isogeny} if $\phi\otimes_\Z\Q\in\Mod(G,R\otimes_\Z\Q)$ is an isomorphism.
\end{defi}

\begin{remark}\label{redtofiniterem}
The category $\Mod(G,R)$ is $R$-linear, and it only depends on the profinite completion $\hat{G}$ of $G$; in fact,
\[
\Mod(G,R)\,=\,\Mod(\hat{G},R)\,=\,\varinjlim_i\Mod(G/G_i,R)\;,
\]
where the $G_i$ vary in the filtered set of normal subgroups of finite index of $G$ and where the inductive limit is an ascending union.
\end{remark}
\begin{defi}
 If $K$ is a field with absolute Galois group $G_K$, the category $\Mod(G_K,R)$ is called the category of \emph{cocharacter modules} with $R$-coefficients over $K$. If $L/K$ is a finite Galois extension whose Galois group is a splitting quotient for an object $W$ of $\Mod(G_K,R)$, we say that $L/K$ is a splitting Galois extension for $W$.
\end{defi}

\begin{remark}\label{torirem}
For any field $K$, the category $\Mod(G_K,\Z)$ is equivalent to the category of $K$-tori.
\end{remark}

The following lemma generalizes \cite{formalnms} Lemma 7.16:

\begin{lem}\label{karoubilem}
Let $G$ be any group, and let $R\hookrightarrow A$ be an injective homomorphism of Dedekind domains of characteristic zero, where $A$ is local and where the image of $R$ in $A$ is dense for the valuation topology. Then the functor $\cdot\otimes_RA$ induces a natural equivalence
\[
\Mod(G,A)\cong(\Mod(G,R)\otimes_RA)^{\textup{pa}}\;,
\]
where $\textup{pa}$ denotes the formation of the pseudo-abelian envelope. Moreover, if $\phi$ is an isogeny in $\Mod(G,A)$, there exists an object $T$ in $\Mod(G,A)$ such that $\phi\oplus\id_T$ extends to an isogeny in $\Mod(G,R)$.
\end{lem}
\begin{proof}
By Remark \ref{redtofiniterem}, we may assume that $G$ is finite. Let $M_A$ be an object of $\Mod(G,A)$, and let us write $S$ and $B$ for the fraction fields of $R$ and $A$ respectively.
Every irreducible factor of $M_B:=M_A\otimes_{A}B$ in $B[G]\textup{-mod}$ is a direct summand of $B[G]=S[G]\otimes_SB$; hence there exist finite modules $V_S$ and $M'_B$ over $S[G]$ and $B[G]$ respectively such that
\[
V_B\,=\,M_B\oplus M'_B
\]
in $B[G]\textup{-mod}$, where we write $V_B:=V_S\otimes_SB$. In other words, we can complement $M_B$ such that the resulting direct sum descends to $S$; it remains to do the same on the integral level. To this end, let now $M'_A$ be an $A[G]$-lattice of $M'_B$; then $M_A\oplus M'_A$ is an $A[G]$-lattice of $V_B$, and it remains to descend it from $A[G]$ to $R[G]$.
Let $(v_i)_{i\in I}$ and $(v'_i)_{i\in I'}$ be finite systems of generators of $M_A$ and $M'_A$ respectively. Since $V_S$ is dense for the valuation topology on $V_B$, there exist systems $(w_i)_{i\in I}$ and $(w_i')_{i\in I'}$ in $V_S$ such that 
\begin{eqnarray*}
v_i-w_i\,\in\,\pi\cdot(M_A\oplus M'_A)&\forall&i\in I\quad\textup{and}\\
v'_i-w'_i\,\in\,\pi\cdot(M_A\oplus M'_A)&\forall&i\in I'\;,
\end{eqnarray*}
where $\pi$ denotes a uniformizer of $R$. In particular, the $w_i$ and the $w_i'$ are contained in $M_A\oplus M'_A$. Let $V_R\subseteq V_S$ be the $R[G]$-submodule generated by the $w_i$ and the $w_i'$. Since the inclusion of $V_A:=V_R\otimes_RA$ into $M_A\oplus M'_A$ reduces to an isomorphism modulo $\pi$, Nakayama's Lemma implies that $V_A=M_A\oplus M'_A$ as $A$-modules and, hence, as $A[G]$-modules, as desired.

We conclude that every $M_A\in\Mod(G,A)$ is a direct summand of an object that is induced from an object in $\Mod(G,R)$ via base change $\cdot\otimes_RA$; that is, it corresponds to a pair $(V_R,\phi)$, where $V_R\in\Mod(G,R)$ and where $\phi$ is an idempotent endomorphism of $V_R\otimes_RA$ in $A[G]\textup{-mod}$. If $N_A\in\Mod(G,A)$ corresponds to a pair $(W_R,\psi)$, then the morphisms $\alpha:M_A\rightarrow N_A$ correspond to the morphisms $\beta:V_{R}\otimes_RA\rightarrow W_{R}\otimes_RA$ such that  that $\beta\phi=\psi\beta=\beta$, which proves the first assertion. 

The assertion regarding isogenies is proved along the same lines, as follows: let $M_A\rightarrow N_A$ be an isogeny; we use it to identify $M_B$ with $N_B$. Let us choose a suitable complement $M'_B$ of $M_B=N_B$ as above as well as an $A$-lattice $M'_A$ for this complement; then after choosing suitable generators as before, we obtain lattices $V_R$ and $W_R$ which descend $M_A\oplus M_A'$ as well as $N_A\oplus M_A'$, and it remains to see that $V_R\subseteq W_R$ within the $S$-space descending $M_B\oplus M_B'=N_B\oplus M_B'$. However, this follows form the fact that $M_A\subseteq N_A$ and, hence, $M_A\oplus M_A'\subseteq N_A\oplus M_A'$.
\end{proof}

\begin{remark}\label{anintrem}
We have already noted that for any field $K$, the category $\Mod(G_K,\Z)$ is equivalent to the category $\Tori_K$ of algebraic $K$-tori. By Lemma \ref{karoubilem}, for any prime number $p$, $\Mod(G_K,\Z_p)$ is thus equivalent to $(\Tori_K\otimes_\Z\Z_p)^\textup{pa}$. If $K$ is complete discretely valued and if the characteristic of the residue field of $K$ coincides with $p$, this category can be given a geometric interpretation within the framework of rigid or uniformly rigid geometry. Let us discuss the uniformly rigid version, cf.\ \cite{formalnms} Lemma 7.14: the category $\Mod(G_K,\Z_p)$ is equivalent to the category $\uTor_K$ of analytic $K$-tori (called uniformly rigid tori in [ibid.]), which is a category of semiaffinoid $K$-groups. Explicitly, if $M\in\Mod(G_K,\Z_p)$ is split by a finite Galois extension $L/K$, let us set
\[
T_L\,:=\,M\otimes_{\Z_p}\hat{\G}_{m,L}\;,
\]
which is an $r$-fold product of copies of the quasi-compact uniformly rigid open unit disc $\hat{\G}_{m,L}$ over $L$ equipped with the multiplicative group structure, where $r=\rank_{\Z_p}M$. The given $G_K$-action on $M$ induces a Galois descent datum on $T_L$; the resulting semi-affinoid $K$-group $T$ (cf.\ \cite{formalnms} Cor.\ 4.16) is the analytic $K$-torus corresponding to $M$.
\end{remark}

\begin{defi}
If $K$ is complete discretely valued and if $\chr\kappa_K=p>0$, we will implicitly use the equivalence discussed in Remark \ref{anintrem}, and we write
\[
X_*(\cdot)\,:\,\uTor_K\rightarrow\Mod(G_K,\Z_p)\quad;\quad T\mapsto X_*(T):=\Hom(T_L,\hat{\G}_{m,L})
\]
for the natural inverse functor. The object $X_*(T)\in\Mod(G_K,\Z_p)$ is called the cocharacter module of $T$. Clearly $\dim T=\rank_{\Z_p}X_*(T)$, where the dimension of a smooth uniformly rigid space is defined in the obvious way.
\end{defi}


Let us observe that if $T$ is an analytic or an algebraic $K$-torus and if $L/K$ is a finite field extension, then 
\[
X_*(T\otimes_KL)\,=\,\Res_{G_L}X_*(T)
\]
is obtained from $X_*(T)$ by restricting the $G_K$-action to $G_L\subseteq G_K$.

\section{The base change conductor}\label{bccondsec}

Let $K$ be a complete discretely valued field with $\chr\kappa_K=p>0$. 
We define the Galois-theoretic base change conductor $c_\Gal(T)$ of $T$ by mimicking de Shalit's recipe for the Lie algebra of an algebraic torus (\cite{chaiyudeshalit} A1.9): we define $c_\Gal(T):=c_\Gal(X_*(T))$ as a sum
\[
c_\Gal(X_*(T))\,:=\,c_\lin(X_*(T))\,+\,c_\nonlin(X_*(T))
\]
of a linear part $c_\lin(X_*(T))$ and a nonlinear part $c_\nonlin(X_*(T))$. Both $c_\lin(\cdot)$ and $c_\nonlin(\cdot)$ are additive with respect to direct sums, and $c_\lin(\cdot)$ is even $\O_K$-linear, that is, an additive invariant on $\Mod(G,\O_K)$. In fact, $c_\lin(\cdot)$ can be defined in a rather general setting. To define $c_\Gal(\cdot)$, we do not use any nonarchimedean analytic geometry whatsoever.

As we show in Section \ref{reltogeomsec}, $c_\Gal(T)$ coincides with the geometric base change conductor $c_\textup{geom}(T)$ of $T$ which can be defined in terms of formal Néron models of formally finite type for uniformly rigid spaces (see \cite{formalnms}); the definition of $c_\Gal(T)$ may be viewed as a recipe to compute $c_{\textup{geom}}(T)$. The reason why we present this recipe in terms of a definition is the following: if one is merely interested in a formula for the base change conductor of a semiabelian $K$-variety with potentially ordinary reduction, then one can work exclusively with $c_\Gal(\cdot)$ and avoid $c_{\textup{geom}}(\cdot)$ altogether; the theory of uniformly rigid spaces and their formal Néron models then disappears from the arguments.

\subsection{The linear base change conductor}\label{linpartsec}

In this subsection, let us pass to the more general situation where $K$ is a henselian valued field. Let $\O_K$ denote the ring of integers of $K$; then each finite field extension $L$ of $K$ carries a unique valuation extending the given valuation on $K$, and the integral closure $\O_L$ of $\O_K$ in $L$ coincides with the ring of integral elements with respect to this valuation. In this context, we define:
\begin{defi}
Let $M$ be in $\Mod(G_K,\O_K)$, let $L/K$ be a finite Galois splitting extension for $M$, let $\Gamma$ denote the Galois group of $M$, and let $e_{L/K}$ denote the ramification index of $L/K$. We let
\begin{enumerate}
\item 
\[
\phi_{M,L}\,:\,(M\otimes_{\O_K}\O_L)^{\Gamma}\otimes_{\O_K}\O_L\rightarrow M\otimes_{\O_K}\O_L
\]
denote the map sending $a\otimes b$ to $a\cdot b$, and we set
\item 
\begin{eqnarray*}
c_\lin(M)&:=&\nu_K({\det}_{\O_L}\phi_{M,L})\\
&=&\frac{1}{e_{L/K}}\length_{\O_L}(\coker\phi_{W,K})\;.
\end{eqnarray*}
\end{enumerate}
The invariant $c_\lin(\cdot)\,:\,\Mod(G_K,\O_K)\rightarrow\Q$ is called the \emph{linear base change conductor}.
\end{defi}

One immediately verifies that $c_\lin(\cdot)$ does not depend on the choice of $L/K$ and that $c_\lin(\cdot)$ is additive with respect to direct sums. Let us note that the domain and the target of $\phi_{M,L}$ depend, as $\O_K$-modules, on the underlying $\O_K[\Gamma]$-module structures of $M$ and $\O_L$ only, while the map $\phi_{M,L}$ itself also depends on the ring structure of $\O_L$.

\subsection{The nonlinear base change conductor}\label{nonlinpartsec}

Let us now assume that the field $K$ is complete discretely valued and that $\chr\kappa_K=p>0$. Then the subgroup $1+t\O_L[[t]]\subseteq\O_L[[t]]^\times$ carries a natural $\Z_p$-module structure. For any finite field extension $L/K$, we consider the natural short exact sequence of $\Z_p$-modules
\[
1\rightarrow 1+t^2\O_L[[t]]\rightarrow 1+t\O_L[[t]]\overset{\red}{\rightarrow}\O_L\rightarrow 0\;,
\]
where $\red$ is the $\Z_p$-linear homomorphism that reduces formal power series in $t$ modulo $t^2$. If $L/K$ is a Galois extension and if $M$ is any object in $\Mod(G_K,\Z_p)$ that is split by $L/K$, then after tensoring with $M$ and taking $\Gamma=\Gal(L/K)$-invariants, we obtain a long exact sequence of $\Z_p$-modules
\[
\xymatrix{
0\ar[r]&(M\otimes_{\Z_p}(1+t^2\O_L[[t]]))^\Gamma\ar[r]& (M\otimes_{\Z_p}(1+t\O_L[[t]])))^\Gamma\ar[r]^(0.63){\textup{red}(M)}&(M\otimes_{\Z_p}\O_L)^\Gamma\\
{\;\;}\ar[r]^(0.2)\partial&\;\;\;H^1(\Gamma,M\otimes_{\Z_p}(1+t^2\O_L[[t]])\ar[r]&\cdots\;.&&
}
\]

In order to define $c_\nonlin(M)$, we need the following maybe surprising fact:

\begin{lem}\label{welldeflem}
For $M\in\Mod(G_K,\Z_p)$ split by a finite extension $L/K$ with Galois group $\Gamma$, $\im\red(M)$ is an $\O_K$-submodule of finite index of $\Hom_{\Z_p}(M,\O_L)^\Gamma$.
\end{lem}
\begin{proof}
If $M=\tilde{M}\otimes_\Z\Z_p$ for an object $\tilde{M}\in\Mod(\Gamma,\Z)$, then the desired statement holds by \cite{chaiyudeshalit} Prop.\ A1.7: indeed, let $T$ denote the $K$-torus corresponding to $\tilde{M}$, let $T^\NR$ denote its Néron model, and let $T_L^\NR$ denote the Néron model of $T_L:=T\otimes_KL$; then $\Hom_{\Z_p}(M,\O_L)^\Gamma$ is canonically identified with the Lie algebra of $(\Res_{\O_L}(T_L)^\NR)^\Gamma$ such that $\im\red(M)$ coincides with the image of the Lie algebra of $T^\NR$ under the smoothening morphism 
\[
T^\NR\rightarrow(\Res_{\O_L}(T_L)^\NR)^\Gamma\;;
\]
hence $\im\red(M)$ is, in this case, an $\O_K$-submodule of finite index, as desired. In general, by Lemma \ref{karoubilem}, we may write $M$ as a direct summand of an object $N$ in $\Mod(G_K,\Z)\otimes_\Z\Z_p$; by what we have shown so far, the desired statement holds for $N$. Let $e$ be an idempotent endomorphism of $N$ such that $M=\ker e$; then $e$ induces idempotent endomorphisms $e_1$ of $\Hom_{\Z_p}(N,1+t\O_L[[t]])^\Gamma$ and $e_2$ of $\Hom_{\Z_p}(N,\O_L)^\Gamma$ respectively that commute with $\red(N)$ such that
\begin{eqnarray*}
\Hom_{\Z_p}(M,1+t\O_L[[t]])^\Gamma&=&\ker e_1\,\\
\Hom_{\Z_p}(M,\O_L)^\Gamma&=&\ker e_2\;.
\end{eqnarray*}
The endomorphism $e_2$ is $\O_K$-linear; indeed, the functor $\Hom$ and the functor of taking $\Gamma$-invariants commute with direct sums. Since $\im\red(N)$ is an $\O_K$-submodule of finite index, it suffices to show that $\im\red(M)$ coincides with $(\im\red(N))\cap\ker e_2$. The inclusion $\subseteq$ is trivial, so let us consider an element $y\in\ker e_2$ for which there exists an element $x$ such that $y=\red(N)(x)$. Then $x'=x-e_1(x)$ lies in the kernel of $e_1$, and it suffices to show that $\red(M)(x')=y$; however, $\red(M)(x')=y-\red(M)(e_1(x))=y-e_2(\red(M)(x))=y-e_2(y)=y$, as desired.
\end{proof}

\begin{remark}\label{welldefrem}
Lemma \ref{welldeflem} seems to be difficult to verify by elementary means, i.e.\ without using the theory of Néron lft-models for algebraic tori. The reason lies in the fact that the invariant subscheme of the Weil restriction of the Néron model of a split torus may be singular, so a description of its arc spaces requires the Néron smoothening process.
\end{remark}

\begin{defi}\label{wildbcdefi}
In the above setting, we now define
\begin{eqnarray*}
c_\nonlin(M)&:=&\length_{\O_K}\Hom_{\Z_p}(M,\O_L)^\Gamma/\im \red(M)\\
&=&\length_{\O_K}\Hom_{\Z_p}(M,\O_L)^\Gamma/\ker\partial\;.
\end{eqnarray*}
\end{defi}

\begin{defi}
In the above setup, if $M$ is an object in $\Mod(G_K,\Z_p)$, then the \emph{Galois-theoretic base change conductor} $c(M)$ of $M$ is defined as the sum
\[
c_\Gal(M)\,:=\,c_\lin(M)\,+\,c_\nonlin(M)\;.
\]
\end{defi}

\begin{defi}\label{liealgdefi}
If $M$ is in $\Mod(G_K,\Z_p)$, we define the Galois-theoretic Lie algebra of the Néron model of $M$ to be
\[
\Lie_\Gal(M^\NR):=\im\red(X_*(T))\;;
\]
if $L/K$ is a finite Galois splitting extension for $M$, then $\phi_{M,L}$ induces a natural $\O_L$-linear morphism
\[
\tilde{\phi}_{M,L}\,:\,\Lie_\Gal(M^\NR)\otimes_{\O_K}\O_L\rightarrow\Lie_\Gal(M_L^\NR)\;,
\]
and we observe that $c_\Gal(M)$ coincides with the $\O_K$-valuation of the $\O_L$-determinant of this map. \end{defi}

Let us emphasize that we defined $\Lie_\Gal(M^\NR)$ without saying anything about the symbol $M^\NR$.

%

\subsection{Relation to geometric base change conductors and isogeny invariance}\label{reltogeomsec}

As above, let us assume that $K$ is complete discretely valued and that $\chr\kappa_K=p>0$. If $T$ is an algebraic $K$-torus, we write $c(T)$ to denote the base change conductor of $T$ which is defined, geometrically, via algebraic Néron models, cf.\ \cite{chaiyudeshalit} \S 10.

\begin{prop}\label{toricompprop}
If $T$ is an algebraic $K$-torus, then
\[
c(T)\,=\,c_\Gal(X_*(T)\otimes_\Z\Z_p)\;,
\]
\end{prop}
\begin{proof}
Indeed, this follows directly from de Shalit's recipe for the Lie algebra of the Néron model of an algebraic $K$-torus, cf.\ \cite{chaiyudeshalit} A1, in particular Prop.\ A1.7.
\end{proof}

\begin{cor}\label{isoginvcor}
If the residue field of $K$ is perfect, then 
\[
c_\Gal(\cdot):\Mod(G_K,\Z_p)\rightarrow\Q
\]
is an isogeny invariant.
\end{cor}
\begin{proof}
By Theorems 11.3 and 12.1 of \cite{chaiyudeshalit}, the base change conductor $c(\cdot)$ is an isogeny invariant on $\Tori_K$ when the residue field of $K$ is perfect. If $\phi:M\rightarrow N$ is an isogeny in $\Mod(G_K,\Z_p)$, then by Lemma \ref{karoubilem} there exists an object $L\in\Mod(G_K,\Z_p)$ such that $\phi\oplus \id_L$ extends to an isogeny in $\Mod(G_K,\Z)$, and with Proposition \ref{toricompprop} we conclude that $c_\Gal(M)+c_\Gal(L)=c_\Gal(N)+c_\Gal(L)$, which implies $c_\Gal(M)=c_\Gal(N)$, as desired.
\end{proof}

\begin{remark}
The seemingly simple invariant $c_\lin(\cdot)$ is \emph{not} an isogeny invariant on $\uTor_K$, as we can see from the following example: let us set $p=2$, $K=\Q_2$ and $L=\Q_2(\sqrt{2})$; then $\Gamma=\Gal(L/K)=\Z/2\Z$, where the action of the nontrivial element $\gamma\in\Gamma$ on $L$ is given by $\sqrt{2}\mapsto -\sqrt{2}$, and $\O_L=\Z_2[\sqrt{2}]$. Let $M=\Z[\Gamma]$ be the regular representation of $\Gamma$ over $\Z$, and let $N=1\cdot\Z\oplus \chi\cdot\Z$ be the sum of its characters, where $\chi(\gamma)=-1$. Then $M$ and $N$ are isogenous, but non-isomorphic over $\Z$ or over $\Z_2$. Since Weil restriction preserves Néron models (cf.\ \cite{blr} 7.6/6), $c_\nonlin(M)=0$, and hence formula \ref{indprop} shows that
\[
c_\lin(M)\,=\,c(M)=\frac{1}{2}\nu_{\Q_2}(\mathfrak{d}_{\Q_2(\sqrt{2})/\Q_2)})\;.
\]
On the other hand,
\[
c_\lin(N)=c(\chi)=\frac{1}{2}\length_{\O_L}(\O_L/\sqrt{2}\O_L)\,=\,\frac{1}{2}\;,
\]
because $(\chi\cdot\Z_2\otimes_{\Z_2}\O_L)^\gamma=\sqrt{2}\O_K$. It follows that $c_\lin(M)\neq c_\lin(N)$; corollary \ref{isoginvcor} now implies that $c_\nonlin(\cdot)$ is not an isogeny invariant either.
\end{remark}

\begin{cor}\label{unrambccor}
If $L/K$ is a finite Galois extension, if $K^\ur\subseteq L$ denotes the maximal subextension of $L/K$ that is unramified over $K$ and if $M\in\Mod(G_K,\Z_p)$, then
\[
c_\Gal(M)\,=\,c_\Gal(\Res_{G_{K^\ur}}M)\;.
\]
\end{cor}
\begin{proof}
By Lemma \ref{karoubilem}, there exist objects $M'\in\Mod(G_K,\Z_p)$ and $N\in\Mod(G_K,\Z)$ such that
\[
M\oplus M'\cong N\otimes_{\Z}\Z_p\;.
\]
Let $T$ be an algebraic $K$-torus with $X_*(T)\cong N$; then by Proposition \ref{toricompprop},
\begin{eqnarray*}
c_\Gal(N\otimes_Z\Z_p)&=&c(T)\quad\textup{and}\\
c_\Gal(\Res_{G_{K^\ur}}N\otimes_\Z\Z_p)&=&c(T\otimes_KK^\ur)\;.
\end{eqnarray*}
Since Néron models of algebraic tori commute with unramified base change, $c(T)=c(T\otimes_KK^\ur)$, so we conclude that
\[
c_\Gal(M)+c(M')\,=\,c_\Gal(\Res_{G_{K^\ur}}M)+c_\Gal(\Res_{G_{K^\ur}}M')\;.
\]
The description of $c_\Gal(\cdot)$ given in Definition \ref{liealgdefi} shows that base change conductors can at most decrease under restriction with respect to unramified field extensions; we thus conclude that in fact $c_\Gal(M)=c_\Gal(\Res_{G_{K^\ur}}M)$, as desired.
\end{proof}

\begin{remark}If the residue field $\kappa_K$ of $K$ is algebraically closed, a semiabelian $K$-variety with potentially ordinary reduction also gives rise to an element of $\Mod(G_K,\Z_p)$. Indeed, if $L/K$ is a finite Galois extension such that $A_L$ has ordinary reduction, then the completion of $(A_L)^\NR$ along its unit section is a split formal torus whose character module $X_*(A_L^\NR|_e)$ is a finite free $\Z_p$-module and, thus, gives rise to an object of $\Mod(G_K,\Z_p)$. De Shalit's recipe \cite{chaiyudeshalit} A1.9, adapted to the construction process for the Néron model of $A$, shows that $c(A)=c_\Gal(X_*(A_L^\NR|_e)$.
\end{remark}

Let us finally clarify the relation between $c_\Gal(\cdot)$ and the geometric base change conductor $c_\textup{geom}(\cdot)$ which is defined via formal Néron models for uniformly rigid spaces. By \cite{formalnms} Cor.\ 7.17, every analytic $K$-torus admits a formal Néron model in the sense of \cite{formalnms} Def.\ 1.1. By \cite{formalnms} Prop.\ 7.2, the formal Néron model of a split uniformly rigid $K$-torus is a product of copies of the formal multiplicative group over $\O_K$ and, hence, is stable under base change. Thus, we can define the geometric base change conductor of an analytic $K$-torus by mimicking the definition which is used in the algebraic context:

\begin{defi}
If $T\in\uTor_K$ is an analytic $K$-torus and if $L/K$ is a finite field extension such that $T_L=T\otimes_KL$ splits, then the \emph{geometric base change conductor} $c_\textup{geom}(T)$ of $T$ is defined to be the rational number
\[
\frac{1}{e_{L/K}}\length_{\O_L}\frac{\Lie(T_L^\NR)}{\Lie(T^\NR)\otimes_{\O_K}\O_L}\;,
\]
where $T^\NR$ and $T_L^\NR$ are the Néron models of $T$ and $T_L$ respectively in the sense of \cite{formalnms} Def.\ 1.1 and where the inclusion of Lie algebras is given by the base change morphism $T^\NR\otimes_{\O_K}\O_L\rightarrow T_L^\NR$ corresponding to the identity on $T_L$.
\end{defi}

\begin{prop}\label{fftcompprop}
Let $T\in\uTor_K$ be an analytic $K$-torus, and let $T^\NR$ denote its formal Néron model in the sense of \cite{formalnms}. Then there is a natural identification 
\[
\Lie(T^\NR)\,\cong\,\Lie_\Gal(X_*(T)^\NR)
\]
such that the base change morphism
\[
T^\NR\otimes_{\O_K}\O_L\rightarrow T_L^\NR\;.
\]
corresponds to the morphism $\tilde{\phi}_{X_*(T),L}$ on the level of Lie algebras. In particular, 
\[
c_\textup{geom}(T)\,=\,c_\Gal(X_*(T))\;;
\]
in other words, the Galois-theoretic base change conductor computes the geometric base change conductor.
%
\end{prop}
\begin{proof}
Let us abbreviate the statement of the proposition by claiming that the 'geometric side' for $T$ is naturally identified with the 'Galois-theretic side'. Let us first consider the case where $T\in\uTor_K$ comes from an algebraic $K$-torus $W\in \Tori_K$, meaning that $X_*(T)\cong X_*(W)\otimes_\Z\Z_p$. Let $L/K$ be a Galois splitting extension for $W$; then by \cite{formalnms} Prop.\ 7.15, the formal Néron models of $T$ and $T_L$ as well as the relevant base change map are obtained from the corresponding data for $W$ via formal completion along closed subgroup schemes of their special fibers. Thus, the geometric sides for $T$ and for $W$ coincide. Moreover, the Galois-theoretic side for $W$ coincides with the geometric side for $W$ by de Shalit's recipe \cite{chaiyudeshalit} A1.9, and the Galois-theoretic sides for $T$ and $W$ coincide for tautological reasons. Hence, the geometric side for $T$ coincides with the Galois-theoretic side for $T$, as desired. In the general case, Lemma \ref{karoubilem} shows that $T$ is a direct factor of the analytic $K$-torus $T'$ attached to an algebraic $K$-torus. By what we have shown so far, the geometric and the Galois-theoretic sides for $T'$ are naturally identified. The geometric and the Galois-theoretic side for $T$ are obtained by splitting off the same relevant direct factor and, hence, coincide as well, which proves the assertion.
\end{proof}

In the following, we will simply write
\[
c(\cdot)\,:=\,c_\Gal(\cdot)
\]
to purge the notation; this is permissible in view of Proposition \ref{fftcompprop}. All properties of $c(\cdot)$ which we will use can be obtained from the Galois-theoretic description of $c(\cdot)$; the geometric characterization of $c(\cdot)$ is, logically, never used in this paper. However, it is important nonetheless: it tells us that $c(\cdot)$ is a canonical invariant which directly comes from geometry.

\subsection{The rational virtual base change character}\label{ratvirtsec}

If $\Gamma$ is a finite group and if $F$ is any field of characteristic zero, we write $R_F(\Gamma)$ to denote the free abelian group that is generated by the set $\Xi(\Gamma,F)$ of the characters of the finitely many isomorphism classes of irreducible $F$-rational representations of $\Gamma$; the elements of $R_F(\Gamma)$ are called the virtual $F$-rational characters of $\Gamma$. We write $R_F(\Gamma)_\Q$ for the finite-dimensional $\Q$-vector space $R_F(\Gamma)\otimes_\Z\Q$; its elements are called the rational virtual $F$-rational characters of $\Gamma$. Clearly $\Xi(\Gamma,F)$ is a $\Q$-basis for $R_F(\Gamma)_\Q$. The character of an $F$-rational representation of $\Gamma$ is a central $F$-valued function on $\Gamma$; hence $R_F(\Gamma)$ and $R_F(\Gamma)_\Q$ are naturally contained in the $F$-vector space $C(\Gamma,F)$ of central $F$-valued functions on $\Gamma$. Let us consider the natural pairing
\[
(\cdot|\cdot)_\Gamma\,:\,C(\Gamma,F)\times C(\Gamma,F)\rightarrow F\,;\,(\phi,\psi)\mapsto\frac{1}{\sharp\Gamma}\sum_{\sigma\in\Gamma}\phi(\sigma)\psi(\sigma^{-1})\;.
\]
By \cite{serre_linres} Prop.\ 32 and its proof, $(\chi|\chi')_\Gamma=0$ for distinct elements $\chi,\chi'\in \Xi(\Gamma,F)$, while $(\chi|\chi)_\Gamma=\dim_{\overline{F}}\End^\Gamma (V_i\otimes_F\overline{F})$, where $V_i$ is the irreducible element of $F[\Gamma]$ corresponding to $\chi_i$ and where $\overline{F}$ denotes an algebraic closure of $F$. In particular, we see that $(\cdot|\cdot)_\Gamma$ restricts to a perfect bilinear pairing
\[
(\cdot|\cdot)_\Gamma\,:\,R_F(\Gamma)_\Q\times R_F(\Gamma)_\Q\rightarrow\Q
\]
and that the basis $\Xi(\Gamma,F)$ of $R_F(\Gamma)_\Q$  is orthogonal with respect to this pairing. 

Let us now consider the case where $F=\Q_p$ and where $\Gamma=\Gal(L/K)$ for some finite Galois extension of our complete discretely valued field $K$, where we assume that $\kappa_K$ is perfect and that $\chr\kappa_K=p>0$. By Corollary \ref{isoginvcor}, the base change conductor $c(\cdot)$ extends from $\Mod(\Gamma,\Z_p)$ to an additive invariant on the monoid of $\Q_p$-rational representations $\Mod(\Gamma,\Q_p)$, hence on the group $R_{\Q_p}(\Gamma)$ generated by this monoid and, hence, to a $\Q$-linear form on $R_{\Q_p}(\Gamma)_\Q$. The pairing $(\cdot|\cdot)_\Gamma$ induces an isomorphism from $R_{\Q_p}(\Gamma)_\Q$ onto the $\Q$-dual space $(R_{\Q_p}(\Gamma)_\Q)^\vee$; hence there exists a unique element $c^\vee$ in $R_{\Q_p}(\Gamma)_\Q$ such that $c(\cdot)=(c^\vee|\cdot)_{\Gamma,F}^\Q$. We call $c^\vee$ the rational virtual base change character.

In the remainder of this paper, we show that $c^\vee=\overline{\bAr}^{(p)}_{L/K}$. By definition, $\overline{\bAr}^{(p)}_{L/K}\in R_{\Q_p}(F)_\Q$, so we have to show that 
\[
c(T)\,=\,(\bAr_{L/K}^{(p)}|\chi_{X^*(T)_{\Q_p}})_\Gamma\quad\quad\quad\quad(*)
\]
or, equivalently, that 
\[
c(T)=(\bAr_{L/K}|\chi_{X^*(T)_{\Q_p}})_\Gamma\;,
\]
where $X^*(T)\in\Mod(G_K,\Z_p)$, called the character module of $T$, denotes the contragredient dual of $X_*(T)$.

\subsection{Behavior of the base change conductor under Weil restriction}\label{weilressec}
By Frobenius reciprocity, the desired equality ($*$) above would imply that the behavior of $c(\cdot)$ under induction on the level of character modules is described by a formula that is an analog of the formula in the statement of Proposition \ref{barresprop}. We shall prove this now, thereby providing a step in the proof of ($*$). As before, let $K$ be a complete discretely valued field with $\chr\kappa_K=p>0$. 
\begin{defi}
If $L/K$ is a finite field extension and if $T\in\uTor_L$, we define $\Res_{L/K}T\in\uTor_K$ via
\[
X^*(\Res_{L/K}T)\,:=\,\Ind^{G_L}_{G_K}X^*(T)\,:=\,\Ind^{G_L/U}_{G_K/U}X^*(T)\;,
\]
where $U\subseteq G_L\subseteq G_K$ is any normal subgroup of finite index acting trivially on $X^*(T)$. 
\end{defi}

\begin{remark}
By definition, the functor $\Res_{L/K}:\uTor_L\rightarrow\uTor_K$ is right adjoint to the base extension functor $\cdot\otimes_KL:\uTor_K\rightarrow\uTor_L$. Let us point out that contrary to the analogous algebraic situation, for $T\in\uTor_L$, the analytic $K$-torus $\Res_{L/K}T$ does \emph{not} necessarily represent the Weil restriction of $T$ in the category of all uniformly rigid $K$-spaces.
\end{remark}

\begin{prop}\label{indprop}
Let $L/K$ be a finite Galois extension, let $M$ be a subfield of $L/K$, and let $T\in\uTor_M$ be split by $L$. Then
\[
c(\Res_{M/K}T)\,=\,f_{M/K}\cdot c(T)\,+\,\frac{1}{2}\nu_K(\mathfrak{d}_{M/K})\cdot\dim T\;.
\]
\end{prop}
\begin{proof}
Let us abbreviate $d:=\dim T$, let $(v_j)_{j=1,\ldots,d}$ be an $\O_M$-basis for $\Lie(T^\NR)$, and let $(\beta_k)_{k=1,\ldots,d}$ be an $\O_L$-basis of
\[
\Lie(T_L^\NR)\,=\,\Hom(X^*(T),\O_L)
\]
such that the matrix $(\alpha)_{jk}$ describing the base change morphisms of $T$ in terms of these bases is upper triangular; such a basis $(\beta_k)_k$ can be found by means of an easy adaption of the Gauss elimination algorithm. The base change conductor of $T$ is given by the $K$-valuation $\nu_K\det(\alpha_{jk})_{jk}$ of the determinant of $(\alpha_{jk})_{jk}$, where $\nu_K=(1/e_{L/K})\nu_L$. 

Let us write $n=[M:K]$, let $(x_i)_{i=1,\ldots,n}$ be an $\O_K$-basis of $\O_M$, and let $(\sigma_l)_{l=1,\ldots,n}$ be a system of representatives of $G_K$ modulo $G_M$. The choice of the $\sigma_l$ induces a direct sum decomposition of $\Ind^{G_K}_{G_M}X^*(T)$ as a $\Z_p$-module. We have natural identifications of $\O_K$-modules
\begin{eqnarray*}
\Hom_{G_K}(\Ind^{G_K}_{G_M}X^*(T),\O_L)&\cong&\Hom_{G_M}(X^*(T),\O_L)\quad\quad\;\textup{and}\\
\Hom_{G_K}(\Ind^{G_K}_{G_M}X^*(T),1+t\O_L[[t]])&\cong&\Hom_{G_M}(X^*(T),1+t\O_L[[t]])\quad,
\end{eqnarray*}
where a homomorphism $\phi$ on the right is identified with $\sigma\otimes\chi\mapsto\sigma\cdot\phi(\chi)$. By Definition \ref{liealgdefi}, these isomorphisms restrict to an identification of $\O_K$-sublattices $\Lie((\Res_{M/K}T)^\NR)\cong\Lie(T^\NR)$ within $\Hom_{G_M}(X^*(T),\O_L)$. In particular, 
\[
(x_iv_j)_{{\tiny\begin{matrix}i=1,\ldots,n\\ j=1,\ldots,d\end{matrix}}}
\]
is an $\O_K$-basis for $\Lie((\Res_{M/K}T)^\NR)$, and the base change map for $\Res_{M/K}T$ is induced by the homomorphism
\[
\Hom_{G_M}(X^*(T),\O_L)\rightarrow\bigoplus_{l=1}^d\Hom(X^*(T),\O_L)
\]
sending $\phi$ to $\sum_{l=1}^d(\chi\mapsto\sigma_l\cdot\phi(\chi))$. Since $(\beta_k)_{k=1,\ldots,d}$ is an $\O_L$-basis for $\Hom(X^*(T),\O_L)$, the same also holds for $(\sigma_l\beta_k)_{k=1,\ldots,d}$ for every $l=1,\ldots,n$. We use the $\O_L$-basis $(\sigma_l\beta_k)_k$ in the $l$-th direct summand of the direct sum above. With respect to these bases, the above base change morphism is described as follows:
\begin{eqnarray*}
x_iv_j&\mapsto&\sum_{l=1}^n(\chi\mapsto\sigma_l(x_iv_j(\chi))\,=\,\sum_{l=1}^n\sigma_l(x_i)\sigma_l(\chi\mapsto v_j(\chi))\\
&=&\sum_{l=1}^n\sigma_l(x_i)\sigma_l\sum_{k=1}^d\alpha_{jk}\beta_k\\
&=&\sum_{l,k}\sigma_l(x_i)\sigma_l(\alpha_{kj})\,\sigma_l\beta_k\quad.
\end{eqnarray*}
Thus we see that the base change conductor of $\Res_{M/K}T$ is given by
\[
c(\Res_{M/K}T)\,=\,\nu_K(\det(\sigma_l(x_i)\sigma_l(\alpha_{jk}))_{(i,j),(k,l)})\quad.
\]
It remains to compute this determinant. Since $(\alpha_{jk})_{jk}$ is upper triangular, we have
\begin{eqnarray*}
\det(\sigma_l(x_i)\sigma_l(\alpha_{jk}))_{(i,j),(k,l)}&=&\prod_{j=1}^d\det(\sigma_l(x_i)\sigma_l(\alpha_{jj}))_{il}\;.
\end{eqnarray*}
By multilinearity of the determinant, this expression equals
\begin{eqnarray*}
&&\prod_{j=1}^d\prod_{l'=1}^n\sigma_{l'}(\alpha_{jj})\det(\sigma_l(x_i))_{il}\\
&=&\Norm_{M/K}(\det(\alpha_{jk})_{jk})\cdot(\det(\sigma_l(x_i))_{il})^d\;,
\end{eqnarray*}
and we conclude that
\begin{eqnarray*}
c(\Res_{M/K}T)&=&\nu_K(\Norm_{M/K}(\det(\alpha_{jk})_{jk})\cdot(\sqrt{\mathfrak{d}_{L/K}})^d)\\
&=&f_{M/K}\cdot\nu_M(\det(\alpha_{jk})_{jk})\,+\,\frac{d}{2}\nu_K(\mathfrak{d}_{M/K})\;,
\end{eqnarray*}
as desired.
\end{proof}

\section{A formula for the base change conductor}\label{formulasec}

Let $K$ be a complete discretely valued field with perfect residue field $\kappa_K$ such that $p:=\chr\kappa_K>0$. Our main result is the following:

\begin{thm}\label{mainthm}
If $L/K$ is a finite Galois extension and if $T\in\uTor_{L/K}$, then
\[
c(T)\,=\,(\bAr_{L/K}|\chi_{X^*(T)_{\Q_p}})_{\Gal(L/K)}\;.
\]
\end{thm}

In the rest of this paper up to Section \ref{abvarapplsec}, we give the proof of Theorem \ref{mainthm}. We easily reduce to the case where $L/K$ is totally ramified: let $\alpha$ denote the inclusion of the inertia subgroup $\Gamma_0$ into $\Gamma:=\Gal(L/K)$, and let $K^\ur=L^{\Gamma_0}$ be the maximal subextension of $L/K$ that is unramified over $K$; then by definition, $\bAr_{L/K}\,=\,\alpha_*\bAr_{L/K^\ur}$, and hence
\[
(\bAr_{L/K}|\chi_{X^*(T)_{\Q_p}})_\Gamma\,=\,(\bAr_{L/K^{\ur}}|\alpha^*\chi_{X^*(T)_{\Q_p}})_\Gamma\,=\,(\bAr_{L/K^\ur}|\chi_{X^*(T\otimes_KK^\ur)_{\Q_p}})_\Gamma\;.
\]
By Corollary \ref{unrambccor}, $c(T)=c(T\otimes_KK^\ur)$; to prove Theorem \ref{mainthm}, we may thus assume that $L/K$ is totally ramified.

\subsection{Reduction to the totally ramified cyclic case}\label{redsec}

By Artin's theorem \cite{serre_linres} Thm.\ 26, $R_{\Q_p}(\Gamma)_\Q$ is generated, as $\Q$-vector space, by elements of the form 
\[
\Ind^\Gamma_{\Gamma'}\chi\;,
\]
where $\Gamma'$ varies among the cyclic subgroups of $\Gamma$ and where $\chi$ varies among the characters of the $\Q_p$-rational representations of $\Gamma'$. To show Theorem \ref{mainthm}, it thus suffices to prove that if $L/K$ is totally ramified, then for any subfield $M$ of $L/K$ such that $L/M$ is cyclic and for any analytic $M$-torus $T$ split by $L$, the statement of Theorem \ref{mainthm} holds, i.e.\ that we have
\[
c(\Res_{M/K}T)\,=\,(\bAr_{L/K}^{(p)}|\chi_{X^*(\Res_{M/K}T)_{\Q_p}})_\Gamma\;.
\]
By Proposition \ref{indprop},
\[
c(\Res_{M/K}T)\,=\,c(T)\,+\,\frac{1}{2}\nu_K(\mathfrak{d}_{M/K})\cdot\dim T\;.
\]
Let us write $\Gamma'=\Gal(L/M)$, let $\alpha:\Gamma'\rightarrow\Gamma$ denote the natural inclusion, and let $\chi:\Gamma'\rightarrow\Q^\alg$ be the character of $X^*(T)_{\Q_p}$; then
\[
\chi_{X^*(\Res_{M/K}T)_{\Q_p}}\,=\,\chi_{\Ind^\Gamma_{\Gamma'}X^*(T)_{\Q_p}}\,=\,\alpha_*\chi\;,
\]
and by Proposition \ref{barresprop},
\[
(\bAr_{L/K}^{(p)},\alpha_*\chi)\,=\,(\alpha^*\bAr_{L/K}^{(p)},\chi)\,=\,\bAr^{(p)}_{L/M}+\frac{1}{2}\nu_K(\mathfrak{d}_{M/K})\textbf{\textup{r}}_{\Gamma'}\;.
\]
It thus suffices to show that
\[
c(T)\,=\,(\bAr^{(p)}_{L/M},\chi)\;;
\]
in other words, after replacing $K$ by $M$, we may assume that $L/K$ is totally ramified and cyclic.

\subsection{The totally ramified cyclic case}\label{cyccasesec}

Let us first discuss the structure of $\Q_p$-rational representations of finite cyclic groups:

\begin{lem}\label{firstcycliclem}
If $\Gamma$ is a finite cyclic group of order $n$ and if $K$ is a field whose characteristic does not divide $n$, then the group algebra $K[\Gamma]$ is the direct product of a collection $(K_i)_{i=1,\ldots,r}$ of finite separable field extensions of $K$, generated by roots of the polynomial $X^n-1\in K[X]$. Moreover, if $V$ is a simple $K[\Gamma]$-module, then $V$ is isomorphic to $K_i$ for some $i$, and $\End_{K[\Gamma]}V\cong K_i$.
\end{lem}
\begin{proof}
Let $n$ denote the order of $\Gamma$; then $K[\Gamma]\cong K[X]/(X^n-1)$ is an étale $K$-algebra, which implies the first statement. The second statement is now clear.
\end{proof}

\begin{lem}\label{cyclicstrlem}
Let $\Gamma$ be a finite cyclic group, and let $p$ be a prime number. Let $\Gamma^w,\Gamma^t\subseteq\Gamma$ denote the subgroups of $p$-torsion and prime-to-$p$-torsion respectively.
\begin{enumerate}
\item If $V^t$ and $V^w$ are simple modules over $\Q_p[\Gamma^t]$ and $\Q_p[\Gamma^w]$ respectively of finite $\Q_p$-dimension, then the $\Q_p[\Gamma]$-module $V^t\otimes_{\Q_p} V^w$ is simple.
\item If $V$ is a simple $\Q_p[\Gamma]$-module of finite $\Q_p$-dimension, then there exist $V^t$, $V^w$ as in $(i)$ such that 
\[
V\cong V^t\otimes_{\Q_P}V^w\;.
\]
\end{enumerate}
\end{lem}
\begin{proof}
Let us first show statement $(i$). Let $n\in\N$ and $p^r$, $r\in\N$, denote the cardinalities of $\Gamma^t$ and $\Gamma^w$ respectively. According to Lemma \ref{firstcycliclem} and its proof, $K^t:=\End_{\Q_p[\Gamma^t]}(V^t)$ and  $K^w:=\End_{\Q_p[\Gamma^w]}(V^w)$ are extension fields of $\Q_p$, where $K^t/\Q_p$ is generated by a root of the polynomial $X^n-1\in\Q_p[X]$, and where $K^w/\Q_p$ is generated by a root of $X^{p^r}-1\in\Q_p[X]$. By \cite{serre_localfields} Chap.\ IV \S 4 Propositions 16 and 17, $K^t/\Q_p$ is unramified, while $K^w/\Q_p$ is totally ramified; in particular, $K^t$ and $K^w$ are linearly disjoint over $\Q_p$, which means that $K^t\otimes_{\Q_p}K^w$ is a field. The claim now follows from \cite{bourbaki_algebreviii} \S 7 no.\ 4 Théorème 2 (b).

Let us now show the second statement. By \cite{serre_linres} Chap.\ 12 Prop.\ 32, the characters of the simple $\Q_p[\Gamma]$-modules form an orthogonal $\Z$-basis of $R_{\Q_p}(\Gamma)$. By statement ({\rm i}), it thus suffices to show that there exist simple modules $V^t$ and $V^w$ over $\Q_p[\Gamma^t]$ and $\Q_p[\Gamma^w]$ respectively such that 
\[
(\chi_V,\chi_{V^t}\otimes\chi_{V^w})\neq 0\;.
\]
And indeed, if all these pairings
\[
((\chi_V(\cdot_1,\cdot_2),\chi_{V^t}(\cdot_1)),\chi_{V^w}(\cdot_2))
\]
were zero, then by \cite{serre_linres} Chap.\ 12 Prop.\ 32 applied first to $R_{\Q_p}(\Gamma^w)$ and then to $R_{\Q_p}(\Gamma^t)$, it would follow that $\chi_V=0$. However, $V$ is simple by assumption and, hence, nonzero.
\end{proof}

Let $K$ be a local field with positive residue characteristic $p$, and let $L/K$ be a totally ramified finite cyclic extension with Galois group $\Gamma$. Then $\Gamma=\Gamma_t\times\Gamma_w$, were $\Gamma_t,\Gamma_w\subseteq\Gamma$ are the subgroups of prime-to-$p$ torsion and $p$-torsion respectively. In particular, $\Gamma_w=\Gamma_1$. We let $n$ and $m=p^r$ denote the cardinalities of $\Gamma_t$ and $\Gamma_w$ respectively; then $p\nmid n$. Moreover, we set $K^t=L^{\Gamma_w}$ and $K^w=L^{\Gamma_t}$. By Corollary \ref{isoginvcor}, $c(\cdot)$ is a function on rational cocharacter modules; let us define $c^*(\cdot)$ as a function on rational character modules by
\[
c^*(X^*(T)_{\Q_p})\,:=c(T)\;,
\]
and let us also set $c_\lin^*(X^*(T)_{\Q_p}):=c_\lin(T)$. Let $T$ be an analytic $K$-torus with rational  character module $X^*(T)_{\Q_p}=V$; by Lemma \ref{cyclicstrlem}, we may assume that $V=V^t\otimes_{\Q_p}V^w$, where $V^t$ and $V^w$ are simple over $\Q_p[\Gamma^t]$ and $\Q_p[\Gamma^w]$ respectively. Let us note that
\[
\Q_p[\Gamma^w]\,=\,\Q_p[X]/(X^{p^r}-1)\;,
\]
and
\[
X^{p^r}-1\,=\,\prod_{i=0}^r\Phi_{p^i}\;,
\]
where $\Phi_{p^i}\in\Q_p[X]$ denotes the $p^i$-th cyclotomic polynomial; it is irreducible in $\Q_p[X]$ and of degree
\[
\deg\Phi_{p^i}\,=\,\phi(p^i)\,=\,(p-1)p^{p^i-1}\,,
\]
cf.\ \cite{serre_localfields} Chap.\ 4 \S 4 Prop.\ 17. By the Chinese Remainder Theorem, the natural projections yields an isomorphism of rings
\[
\Q_p[\Gamma^w]\,\cong\,\prod_{i=0}^r\Q_p[X]/(\Phi_{p^i}(X))\;,
\]
where the factors are fields. Thus $V^w\cong \Q_p[X]/(\Phi_{p^i})$ for some $i\in\{0,\ldots,r\}$. Now
\[
X^{p^i}-1\,=\,\Phi_{p^i}\cdot (X^{p^{i-1}}-1)\;,
\]
and hence again the Chinese Remainder Theorem shows that
\[
\Q_p[X]/(X^{p^i}-1)\,\cong\,\Q_p[X]/(\Phi_{p^i}(X))\times\Q_p[X]/(X^{p^{i-1}}-1)
\]
as $\Q_p[\Gamma^w]=\Q_p[X]/(X^{p^r}-1)$-modules via the natural projections. Hence, it suffices to show the statement of Theorem \ref{mainthm} for the quotients 
\[
V^w=\Q_p[X]/(X^{p^i}-1)
\]
of $\Q_p[\Gamma^w]$, for all $i\in\{0,\ldots,r\}$; of course then $V^w$ needs no longer be simple. Let $\sigma\in\Gamma^w$ be the generator corresponding to the class of $X$ under the chosen isomorphism $\Q_p[\Gamma^w]\cong\Q_p[X]/(X^{p^r}-1)$; then $\sigma^{p^i}$ acts trivially on $V^w$. Let $L'$ denote the subfield of $L$ that is fixed by $\langle\sigma^{p^i}\rangle\subseteq\Gamma^w\subseteq\Gamma$; then $V^t\otimes_{\Q_p}V^w$ is already split by $L'$, and the wild ramification subgroup of $\Gal(L'/K)$ is naturally identified with $\Gamma^w/\langle\sigma^{p^i}\rangle$, which has order $p^i$. Replacing $L$ by $L'$ (which is permissible by Proposition \ref{pushforwardprop}), we thus reduce to the case where $r=i$, i.e.\ where $V^w$ is the regular $\Q_p$-rational representation of $\Gamma^w$. 

Now 
\begin{eqnarray*}
V&\cong& V^t\otimes_{\Q_p}\Q_p[\Gamma^w]\\
&\cong& V^t\otimes_{\Q_p[\Gamma^t]}\Q_p[\Gamma^t]\otimes_{\Q_p}\Q_p[\Gamma^w]\\
&\cong&V^t\otimes_{\Q_p[\Gamma^t]}\Q_p[\Gamma]\\
&\cong&\Ind^\Gamma_{\Gamma^t}V^t\;,
\end{eqnarray*}
and hence we may assume that
\[
T\,=\,\Res_{K^w/K}T^t
\]
for an analytic $K^w$-torus $T^t$ that is split by $L$. By Proposition \ref{barresprop} and Proposition \ref{indprop}, we may thus assume that $K=K^w$, i.e.\ that $L/K$ is cyclic and totally tamely ramified. In this situation, $c(\cdot)=c_\lin(\cdot)$ because then the order $n$ of $\Gamma$ is invertible in any $\Z_p$-module, which implies that the $H^1$ cohomology groups that intervene in the definition of $c_\nonlin(\cdot)$ vanish. Now
\[
c^*_\lin(V)\,=\,c_\lin(V^\vee\otimes_{\Q_p}K)\;,
\]
where $V^\vee$ denotes the contragredient dual of $V$, and after replacing $K$ by an unramified extension if necessary, we may assume that $K$ contains $\mu_n$. Then $V\otimes_{\Q_p}K$ splits as a direct sum of one-dimensional characters, so it suffices to show that
\[
c_\lin(V)\,=\,(\overline{\bAr}_{L/K},V)
\]
for any $\O_K[\Gamma]$-module $V$ that is finite free of rank $1$ over $\O_K$. Let $\sigma$ be a generator of $\Gamma$, and let $\zeta\in\mu_n$ be a primitive $n$-th root of unity; then $\sigma$ acts via multiplication by $\zeta^i$ for a unique $i\in\{0,\ldots,n-1\}$. Let $\pi_L$ be a uniformizer of $L$ such that $\sigma\pi_L=\zeta\pi_L$; such a uniformizer exists by \cite{lang_ant} Chap.\ II \S 5 Prop.\ 12; then
\[
(V\otimes_{\O_K}\O_L)^\Gamma\,=\,\O_L^{\sigma=\zeta^{n-i}}\,=\,\pi_L^{n-i}\O_K\;,
\]
and the base change morphism $\phi_{V,L}$ is the natural inclusion
\[
\pi_L^{n-i}\O_L\rightarrow\O_L\;;
\]
the $\O_L$-length of its cokernel ist $n-i$, so $c(V)=(n-i)/n$. On the other hand, $(\overline{\bAr}_{L/K},\chi_V)=(n-i)/n$, which concludes the proof of Theorem \ref{mainthm}.

\section{Application to semiabelian varieties with potentially ordinary reduction}\label{abvarapplsec}

As before, let $K$ be a local field with $\chr\kappa_K=p>0$. Let $K'$ be the completion of the maximal unramified extension of $K$ in a fixed algebraic closure of $K$, let $A$ be an semiabelian $K$-variety with potentially ordinary reduction, and let us set $A':=A\otimes_KK'$. Then
\[
c(A)\,=\,c(A')\;,
\]
because Néron models commute with formally unramified base change. Let $L/K'$ be a finite Galois extension such that $A_L:=A'\otimes_{K'}L$ has semi-stable reduction, and let $A_L^\NR$ be the Néron model of $A_L$. Then the completion $A_L^\NR|_e$ of $A_L^\NR$ along the unit section of its special fiber is a split formal torus, and $\Gal(L/K')$ acts on its character module $X^*(A_L^\NR|_e)$. By de Shalit's recipe (see \cite{chaiyudeshalit} Prop.\ A1.7),
\[
c(A')\,=\,c^*(X^*(A_L^\NR|_e))\,,
\]
so our main result Theorem \ref{mainthm} yields the following formula for semiabelian $K$-varieties with potentially ordinary reduction:
\begin{thm}\label{mainthmabvar} Let $A$ be a semiabelian $K$-variety; then
\[
c(A)\,=\,(\bAr_{L/K'},\chi_{X^*(A_L^\NR|_e)})_{\Gal(L/K')}\;.
\]
\end{thm}

\subsection{The potentially ordinary CM case}

To illustrate how Theorem \ref{mainthmabvar} may be applied in concrete situations, we give a description of the Galois-module $X^*(A_L^\NR|_e)$ in terms of an algebraic Hecke character, in the special case where the abelian variety under consideration is obtained from a CM abelian variety over a number field via base change to a place of potentially ordinary reduction.

Let $K$ be a number field, let $g$ be a positive integer, let $F$ be a CM-algebra of $\Q$-dimension $2g$, and let $(A,\iota)$ be a CM abelian variety over $K$ with complex multiplication by $F$. Let $\overline{\Q}$ be an algebraic closure of $K$, let $\Phi\subseteq\Hom(F,\overline{\Q})$ denote the CM type of $(A,\iota)$,
and let $F_0\subseteq F$ denote the totally real subalgebra of $F$. Let
\[
\epsilon\,:\,\A_K^\times\rightarrow F^\times
\]
be the algebraic Hecke character that is attached to $A$ by the Main Theorem of Multiplication, see \cite{chai_conrad_oort} Thm.\ 2.5.1. Let $\epsilon_\alg$ denote the algebraic part of $\epsilon$,
and let, for any prime number $\ell$,
\[
\psi_\ell\,:\,\Gal(K^\ab/K)\rightarrow F_\ell^\times
\]
denote the character describing the Galois action on the $\ell$-adic rational Tate module of $A$, where we set $F_\ell:=F\otimes_\Q\Q_\ell$; then by \cite{chai_conrad_oort} Thm.\ 2.5.1 and Lemma 2.4.9 as well as its proof,
\[
\epsilon(x)\,=\,\psi_\ell(r_K(x))\cdot\epsilon_{\alg,\ell}(x_\ell)\;,
\]
where $x\in\A_K^\times$, where $x_\ell$ denotes the $\ell$-component of $x$, where $r_K$ denotes the arithmetically normalized reciprocity map and where we write $\epsilon_{\alg,\ell}$ for 
\[
\epsilon_\alg(\Q_\ell)\,:\,K_\ell^\times\rightarrow F_\ell^\times\;.
\]
For each prime number $\ell$, we choose an algebraic closure $\overline{\Q}_\ell$ of $\Q_\ell$ containing $\overline{\Q}$, and we let $\C_\ell$ denote the completion of $\overline{\Q}_\ell$; then for each $\ell$, we have natural identifications
\[
\Hom(F,\overline{\Q})\,=\,\Hom(F,\overline{\Q}_\ell)\,=\,\Hom(F_\ell,\overline{\Q}_\ell)\,=\,\coprod_{\nu|\ell}\Hom(F_\nu,\overline{\Q}_\ell)\;,
\]
where we might as well write $\C_\ell$ instead of $\overline{\Q}_\ell$. The resulting decomposition of $\Phi$ encodes whether $A$ has potentially ordinary reduction at a place above $\ell$:

\begin{lem}\label{potordredlem}
Let $\nu$ be a place of $K$, and let $\ell$ be the rational prime below $\nu$. Then the following are equivalent:
\begin{packed_enum}
\item The reduction of $A$ in $\nu$ is potentially ordinary.
\item Every place $\mu_0$ of $F_0$ above $\ell$ splits completely in $F$, and if $(\mu,\bar{\mu})$ is a pair of conjugate places of $F$ above $\ell$, then the intersection of $\Phi$ with
\[
\Hom(F_\mu,\overline{\Q}_\ell)\amalg\Hom(F_{\bar{\mu}},\overline{\Q}_\ell)
\]
is either $\Hom(F_\mu,\overline{\Q}_\ell)$ or $\Hom(F_{\bar{\mu}},\overline{\Q}_\ell)$.
\end{packed_enum}
In particular, if $A$ has potentially ordinary reduction in $\nu|\ell$, then $A$ has potentially ordinary reduction in any place $\nu'$ of $K$ dividing $\ell$.
\end{lem}
\begin{proof}
This is a consequence of the Shimura-Taniyama formula: let $K'_\nu/K_\nu$ be a finite Galois extension such that $A\otimes_KK'_\nu$ has semi-stable reduction, let $\sA'$ denote the Néron model of $A\otimes_KK'_\nu$ over $\O_{K'_\nu}$, and let $\kappa$ be an algebraic closure of the residue field of $\O_{K'_\nu}$; then $A$ has potentially ordinary reduction at $\nu$ if and only if the $\ell$-divisible group
\[
\sA'[\ell]\otimes_{\O_{K'_\nu}}\kappa
\]
over $\kappa$ has only multplicative and étale components, i.e.\ slopes $1$ and $0$. The desired statement now follows from \cite{chai_conrad_oort} Lemma 3.7.2.1.
\end{proof}

Let now $\nu$ be a place of $K$ where $A$ has potentially ordinary reduction, let $\ell$ denote the rational prime below $\nu$, and let us choose an embedding of $K_\nu$ into $\overline{\Q}_\ell$ such that the inclusion of $K$ into $K_\nu$ is compatible with the fixed embedding of $\overline{\Q}$ into $\overline{\Q}_\ell$. Let $L$ denote the completion of the maximal unramified extension $K_\nu^\ur$ of $K_\nu$ inside $\C_\ell$; then $L$ is a complete discretely valued field of mixed characteristic $(0,\ell)$ with algebraically closed residue field. Let $\overline{L}$ be an algebraic closure of $L$ inside $\C_\ell$. The natural composite homomorphism
\[
\psi_{\ell,L}\,:\,\Gal(\overline{L}/L)\rightarrow\Gal(\overline{\Q}_\ell/K_\nu^\ur)\hookrightarrow\Gal(\overline{\Q}_\ell/K_\nu)\hookrightarrow\Gal(\overline{\Q}/K)\twoheadrightarrow\Gal(K^\ab/K)\overset{\psi_\ell}{\rightarrow}F_\ell^\times
\]
describes the $\Gal(\overline{L}/L)$-action on the rational Tate module $V_\ell(A_L[\ell^\infty])$ of the $\ell$-divisible group $A_L[\ell^\infty]$ of $A_L$. Let $L'/L$ be a finite Galois extension inside $\overline{L}$ such that $A_{L'}$ has good ordinary reduction, let $\sA'$ denote the Néron model of $A_{L'}$, and let
\[
A_L[\ell^\infty]^0\subseteq A_L[\ell^\infty]
\]
denote the $\ell$-divisible subgroup corresponding to the multiplicative part $\sA'[\ell^\infty]^0$ of the $\ell$-divisible group $\sA'[\ell]$ of $\sA'$. Let us note that $\sA'[\ell^\infty]^0$ is identified with the $\ell$-divisible group of the split formal $\O_{L'}$-torus $\sA'|_e$ which is obtained from $\sA'$ via formal completion along the identity section $e$ of the special fiber of $\sA'$. The character module $X^*(\sA'|_e)$ of this formal torus is a finite free $\Z_\ell$-module of rank $g$, equipped with an action of $\Gal(L'/L)$. Let us observe that the action of $\Gal(\overline{L}/L)$ on $X^*(\sA'|_e)$ differs from the action of $\Gal(\overline{L}/L)$ on the dual of the Tate module $T_\ell(A_L[\ell^\infty]^0)$ of $A_L[\ell^\infty]^0$ precisely by the $\ell$-adic cyclomotic character:

\begin{lem}\label{cyclolem1}
In the above situation, there are natural isomorphisms of $\Z_\ell[\Gal(\overline{L}/L]$-modules
\begin{eqnarray*}
X^*(\sA'|_e)&\cong&T_\ell(A_L[\ell^\infty]^0)^\vee\otimes_{\Z_\ell}T_\ell(\mu_{\ell^\infty})\;,\\
T_\ell(A_L[\ell^\infty]^0)&\cong&X_*(\sA'|_e)\otimes_{\Z_\ell}T_\ell(\mu_{\ell^\infty})\;.
\end{eqnarray*}
\end{lem}
\begin{proof}
The $\Z_\ell[\Gal(\overline{L}/L)]$-module $X^*(\sA'|_e)$ is the projective limit of the $\Z/\ell^n\Z[\Gal(\overline{L}/L)]$-modules
\[
\Hom((\sA'|_e)[\ell^n](\overline{L}),\mu_{\ell^n}(\overline{L}))\;;
\]
here the $\Gal(\overline{L}/L)$-action is given by the natural actions on both $(\sA'|_e)[\ell^n](\overline{L})$ and on $\mu_{\ell^n}(\overline{L})$, where the action on $(\sA'|_e)[\ell^n](\overline{L})$ is defined by taking the natural $\O_{L'}$-semilinear action of $\Gal(L'/L)$ on $\sA'|_e$ into account. Since $(\sA'|_e[\ell^n])\otimes_{\O_L}\O_{L'}\cong\mu_{\ell^n}^g$, the action of $\Gal(\overline{L}/L)$ on $X^*(\sA'|_e)$ is trivial on the open subgroup $\Gal(\overline{L}/L')$ of $\Gal(\overline{L}/L)$. On the other hand, the $\Z_\ell[\Gal(\overline{L}/L)]$-module $V_\ell(A_L[\ell^\infty]^0)$ is the projective limit of the $\Z/\ell^n\Z[\Gal(\overline{L}/L)]$-modules
\[
(\sA'|_e)[\ell^n](\overline{L})\;,
\]
where again the action of $\Gal(\overline{L}/L)$ is defined taking the $\O_{L'}$-semilinear action of $\Gal(L'/L)$ on $\sA'|_e$ into account. This establishes the first isomorphism; the second one is simply the dual of the first.
\end{proof}

Since $A_L$ has CM by $F$, the $\ell$-divisible group $A_L[\ell^\infty]$ has CM by $F_\ell$. Let $\O\subseteq F$ be an order acting on $A_L$; then $\O_\ell$ acts on $A_L[\ell^\infty]$, and by \cite{chai_conrad_oort} Lemma 3.7.2.1, $A_L[\ell^\infty]^0$ is the direct factor of $A_L[\ell^\infty]$ that is cut out by the direct factor
\[
\O_\ell^\Phi\,:=\,\prod_{\nu\in\tilde{\Phi}}\O_\nu
\]
of $\O_\ell$, where $\tilde{\Phi}$ denotes the set of places of $F$ above $\ell$ such that $\Phi$ meets $\Hom(F_\nu,\overline{\Q}_\ell)$, cf.\ Lemma \ref{potordredlem}. Let $F_\ell^\Phi$ denote the direct factor of $F_\ell$ corresponding to $\O_\ell^\Phi$. We know that $V_\ell(A[\infty])$ is free of rank one as a module over $F_\ell$; hence the $\Q_\ell[\Gal(\overline{L}/L)]$-module $V_\ell(A[\infty]^0)$ is given by $\psi_{\ell,L}$, together with the natural action of $F_\ell^\times$ on $F_\ell^\Phi$. 


\begin{lem}
If $A$ has potentially ordinary reduction in the places above $\ell$, then
\[
K_\ell^\times\overset{\epsilon_{\alg,\ell}}{\longrightarrow} F_\ell^\times\rightarrow (F_\ell^\Phi)^\times
\]
is the inverse of the $\ell$-adic cyclotomic character, given by
\[
x\mapsto(\Nm_{K_\ell/\Q_\ell}x)^{-1}\,\in\Q_\ell^\times\subseteq (F_\ell^\Phi)^\times\;.
\]
\end{lem}
\begin{proof}
Since both $\epsilon_{\alg,\ell}$ and the inverse of the cyclotomic character are algebraic and since tori are unirational, it suffices to check the desired equality on an open neighborhood of the identity in $K_\ell^\times$. Since $\epsilon$ is continuous and since
\[
\epsilon\,=\,(\psi_\ell\circ r_K)\cdot\epsilon_{\alg,\ell}\;,
\]
we have
\[
\epsilon_{\alg,\ell}\,=\,(\psi_\ell\circ r_K)^{-1}
\]
on an open neighborhood of the identity in $\A_K^\times$. Let $\nu$ be a place of $K$ above $\ell$, and let $x$ vary in the induced open neighborhood of the direct factor $K_\nu^\times$ of $K_\ell^\times$; then we see that $\epsilon_{\alg,\ell}(x)$ is the inverse of the $\psi_\ell$-image of $r_{K_\nu}(x)\in\Gal(K_\nu^\ab/K_\nu)$ in $\Gal(K^\ab/K)$. The desired equality may be checked after base change from $K_\nu$ to the completed maximal unramified closure $L$ of $K_\nu$, and then again locally near an open neighborhood of the identity of $L^\times$ corresponding to an open neighborhood $\Gal(\overline{L}/L')$ of $\Gal(\overline{L}/L)$ such that $A_{L'}$ has good ordinary reduction. There, the Galois action is indeed given by the cyclotomic character, as we have seen in Lemma \ref{cyclolem1} and its proof, so the statement follows.
\end{proof}

In summary:

\begin{prop}
Let $\epsilon_\nu$ denote the restriction of $\epsilon$ to $K_\nu^\times$; then $\epsilon_\nu$ is continuous, and the $\Q_\ell[\Gal(\overline{L}/L)]$-module $X_*(\sA'|_e))_{\Q_\ell}$ is isomorphic to $F_\ell^\Phi$, where the $\Gal(\overline{L}/L)$-action on $F_\ell^\Phi$ is given by 
\[
K_\nu^\times\overset{\epsilon_\nu}{\longrightarrow}F^\times\rightarrow (F_\ell^\Phi)^\times\;,
\]
by the local arithmetically normalized reciprocity map 
\[
r_{K_\nu}:K_\nu^\times\rightarrow\Gal(K_\nu^\ab/K_\nu)
\]
and by the natural map $\Gal(\overline{L}/L)\rightarrow\Gal(\overline{K}_\nu/K_\nu)$.
\end{prop}

\bibliographystyle{plain}
\bibliography{artin_conductor}

\end{document}

%% file: definitions.tex
\DeclareMathAlphabet{\mathscr}{OT1}{pzc}{m}{it}

\def\romanenum{\renewcommand{\labelenumi}{\textup{(}\roman{enumi}\textup{)}}}

\def\NR{\textup{NR}}

\def\Res{\textup{Res}}

\def\red{\textup{red}}

\def\Q{\mathbb{Q}}
\def\chr{\textup{char }}

\def\ker{\textup{ker }}
\def\im{\textup{im }}
\def\coker{\textup{coker }}

\def\Hom{\textup{Hom}}

\def\A{\mathbb{A}}
\def\O{\mathcal{O}}

\def\Z{\mathbb{Z}}
\def\N{\mathbb{N}}

\def\G{\mathbb{G}}

\def\m{\mathfrak{m}}

\def\C{\mathbb{C}}

\def\plim{\varprojlim}
\def\R{\mathbb{R}}

\def\id{\textup{id}}
\def\sep{{\textup{\footnotesize sep}}}
\def\Gal{\textup{Gal}}

\def\GL{\textup{GL}}

\def\Lie{\textup{Lie}}
\def\Aut{\textup{Aut}}

\def\length{\textup{length}}

\def\mod{\;\textup{mod}\;}

\def\rank{\textup{rank}\,}

\def\ur{\textup{ur}}

\def\alg{\textup{alg}}

\def\Tori{\textup{Tori}}
\def\End{\textup{End}}

\def\bA{\textup{bAr}}
\def\Ind{\textup{Ind}\,}

\def\Inf{\textup{Inf}\,}
\def\Ar{\textup{Ar}}

\def\Tori{\textup{Tori}}

\def\Mod{\textup{Mod}}
\def\uTor{\textup{aTori}}
\def\aTor{\textup{aTori}}
\def\Tori{\textup{Tori}}

\def\Norm{\textup{Norm}}

\def\lin{\textup{lin}}
\def\nonlin{\textup{nonlin}}
\def\bAr{\textup{bAr}}
\def\ab{\textup{ab}}
\def\Nm{\textup{Nm}}
\def\sA{\mathcal{A}}